\documentclass [12pt]{article}

\usepackage{amssymb}
\usepackage{amsfonts}
\usepackage{graphicx}
\usepackage{amsmath}

\usepackage{lineno}
\usepackage{textcomp}
\usepackage{eurosym}
\usepackage{sidecap}
\usepackage{here}

\newtheorem{theorem}{Theorem}[section]

\newtheorem{assumption}[theorem]{Assumption}

\newtheorem{claim}[theorem]{Claim}

\newtheorem{lemma}[theorem]{Lemma}

\newtheorem{proposition}[theorem]{Proposition}

\newenvironment{proof}[1][Proof]{\noindent\textit{#1.} }{\hfill \rule{0.5em}{0.5em}}

\newcommand{\R}{\mathbb{R}}

\newcommand{\C}{\mathbb{C}}

\renewcommand{\eqref}[1]{{\rm(\ref{#1})}}
\renewcommand{\d}{{\rm d}}

\usepackage[utf8]{inputenc} 

\usepackage{xcolor}

\title{\textsc Optimal intervention strategies of staged progression HIV infections through an age-structured model with probabilities of ART drop out
}
\author{Mboya Ba$^{a,*}$, Ramsès Djidjou-Demasse$^{b,\ddagger}$, Mountaga Lam$^{a}$,\\
	Jean-Jules Tewa$^{c}$ \\
	{\small $^{a}$ University Cheikh Anta Diop, Department of Mathematics and Informatics}\\
	{\small Faculty of Science and Technic, Dakar, Sénégal}\\
	{\small $^{b}$ MIVEGEC, IRD, CNRS, Univ. Montpellier, Montpellier, France}\\
	{\small $^{c}$ University of Yaounde I, National Advanced School of Engineering, Yaoundé, Cameroon}\\
	{\small $^*$ UMI 209 IRD\&UPMC UMMISCO}\\
	{\small $^\ddagger$ Author for correspondence: ramses.djidjou-demasse@umontpellier.fr}
}

\begin{document}
\maketitle

\begin{abstract} 
	In this paper, we construct a model to describe the transmission of HIV in a homogeneous host population. By considering the specific mechanism of HIV, we derive a model structured in three successive stages: (i) primary infection, (ii) long phase of latency without symptoms and (iii) AIDS. Each HIV stage is stratified by the duration for which individuals have been in the stage, leading to a continuous age-structure model. In the first part of the paper, we provide a global analysis of the model depending upon the basic reproduction number $\mathcal{R}_0$. When $\mathcal{R}_0\le1$, then the disease-free equilibrium is globally asymptotically stable and the infection is cleared in the host population. On the contrary, if $\mathcal{R}_0>1$, we prove the epidemic's persistence with the asymptotic stability of the endemic equilibrium. By performing the sensitivity analysis, we then determine the impact of control-related parameters of the outbreak severity. For the second part, the initial model is extended with intervention methods. By taking into account ART interventions and the probability of treatment drop out, we discuss optimal interventions methods which minimize the number of AIDS cases.
\end{abstract}



\section{Introduction}
\paragraph{Biology and evolution of HIV infection.}
The human immunodeficiency virus (HIV) is a virus that attacks the immune system, the body's defense against infections. HIV weakens your immune system by destroying cells that are essential for fighting diseases and infections. Without treatment, the immune system becomes too weak. A chronic progressive disease called AIDS (Acquired Immunodeficiency Syndrome) then appears. The situation of the epidemic in the world shows only a stabilization of the number of new cases diagnosed, although undeniable efforts have been made in recent years \cite{onusida}. The complexity of HIV infection is linked to many elements that particularly involve the specific mechanism of infection \cite{destruction}.  In the absence of treatment, the HIV infection goes through three successive stages corresponding to T4 cell count ranges: (i) primary infection, (ii) long phase of latency without symptoms, and (iii) AIDS \cite{OMS}. {\it Primary infection:} The risk of transmission is particularly high during this phase because of the high viral load at this stage of the infection \cite{pathology,Hollingsworth2008,Steven}. This stage is characterized by occurrence of symptoms similar to those of a cold or a mild influenza (fever, rashes, fatigue, headaches) which disappear spontaneously after few weeks \cite{Swiss,pathology}. {\it Phase of latency without symptoms:}
Generally, HIV-positive people do not experience any particular problem at this stage for many years and can lead normal lives, although the virus is spreading insidiously in the body and permanently mistreats the immune system \cite{Swiss,pathology}.
{\it AIDS stage:} Because of its constant solicitation, the immune system becomes weaker and weaker until it can no longer defend itself against many pathogens agent and prevent the occurrence of serious or fatal diseases \cite{Swiss, pathology}.

\paragraph{Waiting time within HIV infection stages.} Staged progression models have been proposed to investigate the transmission dynamics of HIV \cite{PerelsonAndNelson1999,Lin,Hyman,PerelsonAndNelson1999,McCluskey,Guo2011,Gumel2006,Eaton2014,Hyman-Sanley1999}. Here, we go through the same direction by  modeling the progression through three HIV stages described previously (which are enough for the practical interpretation of the HIV stages up to date). However, none of the aforementioned works deal with a continuous stage-structured model as in the context of this work. Indeed, here we consider the duration $a\ge0$ spent in a given HIV stage as a continuous variable (not to be confused with the time since infection as in \cite{Shen2015}, or with age-group such as "youths", "adults", etc). Further, our model formulation is well adapted for the dynamics of HIV infection without any treatment: reference values for duration of HIV stages 1 and 2 are 2.90 (range 1.23-6) and 120 (range 108-180) months respectively \cite{Hollingsworth2008}. Moreover, the model proposed here is suitable for the specific mechanisms of antiretroviral therapy, ART for short, which is the use of HIV medicines to slow down the progress of the infection. In fact, ART help people with HIV live longer by extending the time spent in a given HIV stage \cite{info}.


\paragraph{The model.}
Here we formulate a model structured by the duration $a$ for which individuals have been in a given HIV stage. The model is called {\it age-structured} model thereafter. The host population is divided into four compartments: $S(t)$ denotes the density of susceptible individuals at time $t$, while $i_1(t,a)$, $i_2(t,a)$ and $i_3(t,a)$ respectively denote the density of the primary infected, infected in asymptomatic stage and infected in AIDS stage at time $t$ who have been in that stage for duration $a>0$ (the time spent in the stage). The transfer diagram reads:
\begin{figure}[H]
	\includegraphics[width=1\textwidth]%
	{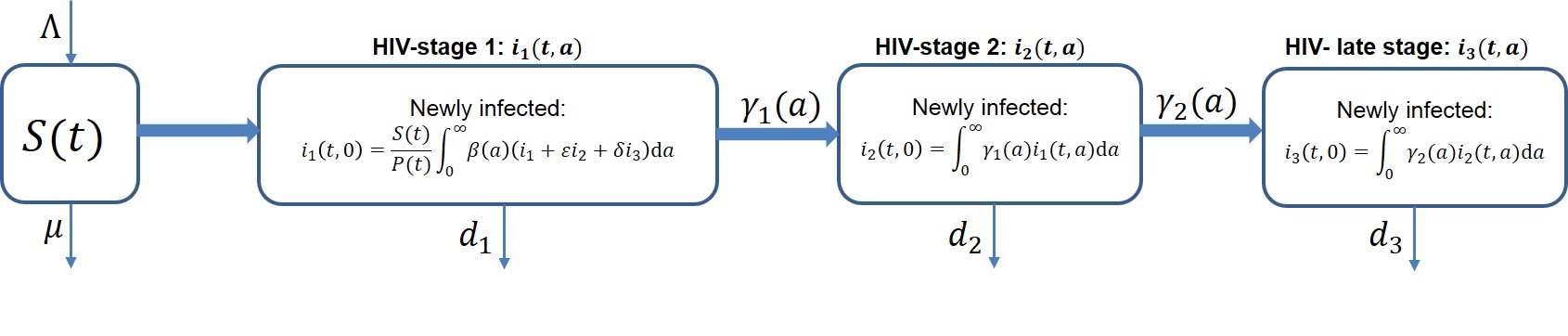}
\end{figure}\vspace{-.8cm}
The parameter $\Lambda$ represents the positive contribution entering into the susceptible population and $\mu$ is the natural death rate of susceptible individuals. The function $\beta= \beta(a)$ is the rate of being infectious after a time $a$ within HIV stage 1, $\varepsilon>0$ and $\delta\ge 0$ are the reduced transmission rate in HIV stages 2 and 3. The death rate of infectious individuals $i_1$, $i_2$ and $i_3$,  are respectively denoted by functions $d_1=d_1(a)$, $d_2=d_2(a)$ and $d_3=d_3(a)$. Obviously, infectious individuals are assumed to have an increased rate of death ({\it i.e.} $d_j\ge \mu$). The rate of the disease progression from infectious class $i_1$ to the infectious class $i_2$ is $\gamma_1=\gamma_1(a)$, as well as $\gamma_2= \gamma_2(a)$ the rate of progression from infectious class $i_2$ to the infectious class $i_3$.

The model we shall consider in this work reads as the following age structured system of equations, 
\begin{equation}
\begin{cases}
\dot S(t)=\Lambda-\mu S(t)-\frac{S(t)}{P(t)} \int_0^\infty \beta(a) \left(i_1(t,a)+\varepsilon i_2(t,a) +\delta i_3(t,a)\right) \d a, \\
i_1(t,0)= \frac{S(t)}{P(t)}  \int_0^\infty \beta(a) \left(i_1(t,a)+\varepsilon i_2(t,a) +\delta i_3(t,a)\right) \d a,\\
\left(\partial_t+\partial_a\right) i_1(t,a)= - \left(\gamma_1(a)+ d_1(a)\right) i_1(t,a),\\
i_2(t,0)=\int_0^\infty \gamma_1(a) i_1(t,a) \d a,\\
\left(\partial_t+\partial_a\right) i_2(t,a)= - \left(\gamma_2(a)+ d_2(a)\right) i_2(t,a),\\
i_3(t,0)=\int_0^\infty \gamma_2(a) i_2(t,a) \d a,\\
\left(\partial_t+\partial_a\right) i_3(t,a)= -  d_3(a) i_3(t,a),\\
\text{with}\\
P(t)= S(t) + \int_0^\infty \left(i_1(t,a)+i_2(t,a)+i_3(t,a) \right)\d a,
\end{cases}\label{model1}
\end{equation}
coupled with the initial condition
\begin{equation}\label{eq-initial}
S(0)=S_0, i_1(0,a)=i_{10}(a), i_2(0,a)=i_{20}(a), i_3(0,a)=i_{30}(a).
\end{equation}

Furthermore, model parameters are assumed to satisfy the following hypotheses:
\begin{assumption}:\label{asym1}
	\begin{enumerate}
		\item $\Lambda,\mu>0$; $\beta , d_1 , d_2, d_3 , \gamma_1, \gamma_2,  \in L_+^\infty (0, \infty)$ and $d_j(\cdot) \ge \mu$.
		\item $\beta$ and $\gamma_j$ are Lipschitz continuous almost everywhere on $\R_+$.
		\item For any $a > 0$, there exists $a(\beta); a(\gamma_j)
 > a$, with $j=1,2,3$; such that $\beta$ is positive in a neighbourhood of $a(\beta)$ and $\gamma_j$ is positive in a neighbourhood of $a(\gamma_j)$.
	\end{enumerate}
\end{assumption}

\paragraph{Aims.} In the first part of this article we shall study some dynamical properties of problem \eqref{model1}-\eqref{eq-initial}. We shall first compute the basic reproduction number which determines the outcome of the disease and study the stability of the model equilibria. We shall perform the global sensitivity analysis of the HIV late stage.  That is to help us know parameters that are most influential in determining disease dynamics, {\it i.e.} AIDS cases.

In the second part of this article, we then introduce intervention strategies into model \eqref{model1} aiming to optimally reduce AIDS cases in the host population.  Several HIV intervention options (called controls) do exist. Individuals do not take the same combination of medicines because the infection stages differ \cite{info}. Therefore, interventions strategies introduce in the model of this note are stage specific. Furthermore, there have been numerous works on optimal control of age-structured populations \cite{Anita2000, DaPrato-Iannelli1994, Barbu-Iannelli1999, Fister,Lenhart2007,Ram,SilvaFilho2018,Ekeland1974,Feichtinger2003,Numfor2015} and references cited therein.

This paper is organized as follows. Section \ref{prelim-first-model} describes preliminaries results of model \eqref{model1}: existence of semiflow and asymptotic behaviour. We derive the global sensitivity analysis, describe model parameters and the typical model simulation in Section \ref{sec-parameters-simu-SA}. In Section \ref{sec-optimal-model}, we extend model \eqref{model1} with intervention strategies and characterize the necessary optimality condition. We then discuss the effectiveness of those interventions and some model hypothesis and limitations in Section \ref{sec-discussion}. Sections \ref{sec-proof-semiflow}-\ref{sec-control-existence} are devoted to proofs of our main results.

\section{Preliminaries} \label{prelim-first-model}
The aim of this section is to provide some preliminary remarks to system \eqref{model1}. Let us introduce the following notations, for $a\ge 0$, 
\begin{equation*}\label{eq-Dj}
\begin{split}
& D_3(a)= \exp \left(-\int_0^a \left(d_3(\sigma)\right) \d \sigma\right), D_j(a)= \exp \left(-\int_0^a \left(\gamma_j(\sigma) + d_j(\sigma)\right) \d \sigma\right); j=1,2,\\
& \overline{D}_j= \int_0^\infty D_j(a)\d a, \quad \Omega_j= \int_0^\infty \beta(\sigma) D_j(\sigma) \d \sigma;\quad j=1,2,3,\\
&\Gamma_j= \int_0^\infty \gamma_j(\sigma) D_j(\sigma) \d \sigma;\quad  j=1,2.
\end{split}
\end{equation*}

\subsection{Existence of semiflow}
We first formulate system \eqref{model1} composed by $ \left(S,i_1,i_2,i_3\right)$ in an abstract Cauchy problem. For that aim, we introduce the Banach space $\mathcal{X}=\mathbb{R} \times\mathbb{R}^3\times L^1(0,\infty,\mathbb{R}^3)$ endowed with the usual product norm $\|\cdot\|_{\mathcal{X}}$ as well as its positive cone $\mathcal{X}_+$. Let $A:D(A)\subset \mathcal{X} \rightarrow \mathcal{X}$ be the linear operator defined by $D(A)  =  \mathbb{R}\times \{0_{\mathbb{R}^3}\}\times W^{1,1}(0,\infty, \mathbb{R}^3)$ and 
\begin{equation*}\label{op-A}
A\left(
v,0,0,0,
u_1,u_2,u_3
 \right)^T
=\left(
\begin{array}{l}
-\mu v\\
-u_1(0)\\-u_2(0)\\ -u_3(0)\\
-u'_1-(\gamma_1+d_1)u_1\\-u'_2-(\gamma_2+d_2)u_2\\
-u'_3-d_3u_3
\end{array}
\right).
\end{equation*}
Finally, let us introduce the non-linear map $F:\overline{D(A)}\rightarrow \mathcal{X}$ defined by

$$
F\left(
v,0,0,0,
u_1,u_2,u_3
\right)^T
=\left(\begin{array}{c}
\Lambda- \frac{v}{v+ \sum_{j=1}^3\int_0^\infty u_j(a)\d a } \int_0^\infty \beta(a) \left(u_1(a)+\varepsilon u_2(a) +\delta u_3(a) \right) \d a \\
\frac{v}{v+ \sum_{j=1}^3\int_0^\infty u_j(a)\d a } \int_0^\infty \beta(a) \left(u_1(a)+\varepsilon u_2(a) +\delta u_3(a)\right) \d a\\
\int_0^\infty \gamma_1(a) u_1(a) \d a\\
\int_0^\infty \gamma_2(a) u_2(a) \d a\\
0\\
0\\
0
\end{array}\right).
$$

By identifying $\varphi(t)$ together with $(S(t), 0_{\mathbb{R}^3}, i_1(t,.),i_2(t,.),i_3(t,.))^T$ and by setting $\varphi_0=(S_0, 0_{\mathbb{R}^3}, i_{10}(.),i_{20}(.),i_{30}(.))^T,$ system \eqref{model1} rewrites as the following Cauchy problem 
\begin{equation}\label{abstract}
\left\{\begin{array}{l}
\displaystyle\frac{d \varphi(t)}{dt}= A\varphi(t)+F(\varphi(t)), \vspace{0.2cm} \\
\varphi(0)= \varphi_0.
\end{array}\right.
\end{equation}

By setting $\mathcal{X}_0 = D(A)$ and $\mathcal{X}_{0+} = \mathcal{X}_0 \cap \mathcal{X}_+$ the precise result is the following theorem.
\begin{theorem}\label{thm-semiflow}
Let Assumption \ref{asym1} be satisfied. Then there exists a unique strongly
continuous semiflow $\left\{ \Phi(t,\cdot) : \mathcal{X}_0 \to \mathcal{X}_0\right\}_{t\ge 0}$ such that for each $\varphi_0 \in \mathcal{X}_{0+}$, the map $\varphi \in \mathcal{C} \left((0,\infty), \mathcal{X}_{0+} \right)$ defined by $\varphi=\Phi(\cdot,\varphi_0)$ is a mild solution of \eqref{abstract}, namely, it
satisfies $\int_0^t \varphi(s)\d s \in D(A)$ and $\varphi(t)= \varphi_0 + A \int_0^t \varphi(s)\d s + \int_0^t F\left(\varphi(s)\right)\d s$ for all $t\ge 0$. Moreover, $\left\{ \Phi(t,\cdot) \right\}_t$ satisfies the following properties:
\begin{enumerate}
\item Let $\Phi(t,\varphi_0)=(S(t), 0_{\mathbb{R}^3}, i_1(t,\cdot),i_2(t,\cdot),i_3(t,\cdot))^T; $ then the following Volterra formulation holds true
\begin{equation} \label{eq-volterra}
i_1(t,a)= \begin{cases}
i_{10}(a-t) \frac{D_1(a)}{D_1(a-t)}, \quad \text{ for } t<a,\\
\frac{S(t-a)}{P(t-a)}E_1(t-a) D_1(a), \quad \text{ for } t\ge a
\end{cases}
\end{equation}

\begin{equation*}
\begin{split}
i_j(t,a)= \begin{cases}
i_{j0}(a-t) \frac{D_j(a)}{D_j(a-t)}, \quad \text{ for } t<a,\\
E_j(t-a) D_j(a), \quad \text{ for } t\ge a,
\end{cases} & j=2,3;
\end{split}
\end{equation*}	
wherein 
\begin{equation}\label{eq-E1-E2}
\begin{split}
& E_1(t)= \int_0^\infty \beta(a) \left(i_1(t,a)+\varepsilon i_2(t,a) +\delta i_3(t,a)\right) \d a, \text{ and }\\ 
&E_{j+1}(t)= \int_0^\infty \gamma_j(a) i_j(t,a) \d a, \quad j=1,2.
\end{split}
\end{equation}		
\item For all $\varphi_0 \in \mathcal{X}_{0+}$ one has for all $t\ge 0$,
\begin{equation*}
\| \Phi(t,\varphi_0) \|_{\mathcal{X}} \le \max \left\{ \frac{\Lambda}{\mu}, \frac{\Lambda}{\mu}+ e^{-\mu t} \left( \|\varphi_0 \|_{\mathcal{X}} - \frac{\Lambda}{\mu} \right) \right\} \le \max \left\{ \frac{\Lambda}{\mu},  \|\varphi_0 \|_{\mathcal{X}} \right\}.
\end{equation*}
\item The semiflow $\left\{ \Phi(t,\cdot) \right\}_t$ is bounded dissipative and asymptotically smooth.
\item There exists a nonempty compact set $\mathcal{B} \subset \mathcal{X}_{0+}$ such that 
\subitem(i) $\mathcal{B}$ is invariant under the semiflow $\left\{ \Phi(t,\cdot) \right\}_t$.
\subitem(ii) The subset $\mathcal{B}$ attracts the bounded sets of $\mathcal{X}_{0+}$ under the semiflow $\left\{ \Phi(t,\cdot) \right\}_t$.
\end{enumerate}
\end{theorem}
We refer to Section \ref{sec-proof-semiflow} for the proof of Theorem \ref{thm-semiflow}.

\subsection{Basic reproduction number and asymptotic behaviour}
An equilibrium $(S,i_1(a),i_2(a),i_3(a))$ of system \eqref{model1} is such that $i_j(a)=D_j(a) i_j(0)$, with $j=1,2,3$; and 
\[
\begin{cases}
i_1(0)=\frac{S}{P} \left(i_1(0)  \Omega_1 +\varepsilon i_2(0)  \Omega_2 + \delta i_3(0)  \Omega_3\right),\\
i_2(0)= \Gamma_1 i_1(0),\\
i_3(0)= \Gamma_2 i_2(0),\\
i_1(0)=\Lambda-\mu S,\\
P=S+ \sum_{j=1}^3 \overline{D}_ji_j(0) .
\end{cases}
\]
The disease-free equilibrium corresponds to $i_1(0)=i_2(0)=i_3(0)=0$ and is given by $E^0= \left(\frac{\Lambda}{\mu},0,0,0\right)$.

In order to find any endemic equilibria, we first determine the basic reproduction number $R_0$ using the next generation operator approach \cite{Diekmann1990,Inaba2012}. We calculate (see Section \ref{basicreproduction})
\begin{equation*}
R_0=  \Omega_1 +\varepsilon \Omega_2\Gamma_1+\delta\Omega_3\Gamma_1\Gamma_2.
\end{equation*}

Now by taking $i_j(0)$'s positive, by straightforward algebra, the disease-endemic equilibrium $E^*=\left(S^*,i_1^*(\cdot),i_2^*(\cdot), i_3^*(\cdot)\right)$ is such that 
\[
\begin{split}
&S^*= \frac{I^*_0 \overline{D} }{R_0-1};\:\: i_1^*(\cdot)= I^*_0   D_1(\cdot);\:\: i_2^*(\cdot)=  I^*_0  \Gamma_1 D_2(\cdot) ;\:\: i_3^*(\cdot)=  I^*_0  \Gamma_1 \Gamma_2 D_3(\cdot),
\end{split}
\] 
with $\overline{D}= \overline{D}_1+ \Gamma_1 \overline{D}_2+ \Gamma_1\Gamma_2\overline{D}_3$ and $I^*_0= \frac{\Lambda (R_0-1) }{\mu \overline{D}+R_0-1}$.
From where, the following proposition summarizes the equilibria of the model.
\begin{proposition}
Let Assumption \ref{asym1} be satisfied. Then the semiflow $\left\{ \Phi(t,\cdot) \right\}_t$  provided by Theorem \ref{thm-semiflow} has exactly:
\begin{itemize}
\item [(i)] One equilibrium, the disease-free equilibrium $E^0$, if $R_0\le 1$.
\item [(i)] Two equilibria, the disease-free equilibrium $E^0$ and disease-endemic equilibrium $E^*$, if $R_0> 1$.
\end{itemize}
\end{proposition}

We end this section by providing the following result concerning the asymptotic behaviour of model \eqref{model1} with respect to the $R_0$.
\begin{theorem} \label{Theo-asymptotic}
Let Assumption \ref{asym1} be satisfied. 
\begin{enumerate}
\item The disease-free equilibrium $E^0$ is globally asymptotically stable if $R_0 < 1$ and unstable if $R_0 > 1$.
\item When $R_0>1$, then
\subitem (i) The disease-endemic equilibrium $E^*$ is locally asymptotically stable for $\{\Phi(t,\cdot) \}_{t}$.
\subitem (ii) If the initial conditions $y_0$ satisfy $i_{10}=i_{20}=i_{30}=0$, then the semiflow $\{\Phi(t,y_0) \}_{t}$ tends to the disease-free equilibrium $E^0$ for the topology of $\mathcal{X}$.
\subitem (iii) If the initial conditions $y_0$ satisfy $i_{10}+i_{20} +i_{30}>0$, then the semiflow $\{\Phi(t,y_0) \}_{t}$ is uniformly persistent in the sense that there exists $\nu>0$ such that 
\begin{equation*}
\liminf_ {t \to \infty} \int_0^\infty \beta(a) \left(i_1(t,a) +\varepsilon i_2(t,a)+\delta i_3(t,a) \right) \d a \ge \nu.
\end{equation*}
\end{enumerate}
\end{theorem} 
We refer to Section \ref{sec-proof-thm24} for the proof of Theorem \ref{Theo-asymptotic}.

\section{Model parameters, Typical model simulation and Sensitivity analysis} \label{sec-parameters-simu-SA}

\subsection{Setting model parameters}\label{sec-para1}
In this section we briefly describe the shape and the values of parameters  consider for the simulations of model \eqref{model1}. For all simulations, parameters $\Lambda$ and $\mu$ are assumed to be fixed with constant values given in Table \ref{tab-model-parameters}. We also assume that the disease induced mortality for HIV stage $j$, $d_j(a)$, is constant with respect to the duration within the stage ({\it i.e.} $d_j(a)\equiv d_j$) and the fixed reference value is given in Table \ref{tab-model-parameters}. Therefore, we more specifically describe the duration-dependent parameters $\beta(a)$ and $\gamma_j(a)$.  
\paragraph{Transmission rate $\beta$.} The parameter $\beta$ is defined for each stage of infection as in \cite{Hollingsworth2008,Eaton2014}. For simplicity, we assume that individuals with AIDS do not substantially contribute in further spread of HIV ({\it i.e.}, not risky sexual behavior) such that we can set $\delta=0$. The transmission rate is assumed to be constant for each stage of infection. During stage $j$, the transmission rate is $\rho_0\times \beta_j$ with $j=1,2$. Here, $\rho_0>0$ is the rate at which an infectious individual infects the susceptible individuals, and $\beta_j$'s are transmission hazard. As pointed in \cite{Hollingsworth2008,Eaton2014}, HIV is 26 times more infectious during stage 1 than during stage 2. For simulations, the reference values and range of parameters $\beta_1$ and $\beta_2$ are given in Table \ref{tab-model-parameters}. Further, we set the transmission rate $\rho_0\simeq 2.48$, such that the maximum value of $R_0$ for the set of variation of parameters is 7.
\paragraph{Rate of progression from HIV stage $j$ to stage $j+1$, $\gamma_j(a)$.} The parameter $\gamma_j$ is set to 
\[
\gamma_j(a) = 0 \text{ if } a< T_0^j \quad \text{and} \quad \bar \gamma_j \text{ if } a \ge T_0^j;
\]
wherein $T_0^j$ is the duration of the $j$-stage of HIV infection and $\bar \gamma_j=1$ is the constant progression rate.  Infected individuals in stage $j$ remain in that stage for a period of time $T_0^j$ and then progress to stage $j+1$ at a constant rate $\bar \gamma_j$ after the duration of the stage $T_0^j$. For simulations, the reference values and the range of parameters $\bar \gamma_j$ and $T_0^j$ are given in Table \ref{tab-model-parameters}.




\subsection{Typical epidemic dynamics simulated with the model}\label{sec-initial1}
For all simulations consider in this note, we assume that the initial susceptible population is $S(0)=\Lambda/\mu$ and the initial distribution of infectives are $i_{10}(a)= c_{01}l_{10}(a)$, $i_{20}(a)= c_{02}l_{20}(a)$ and $i_{30}(a)=0$ for all $a\ge 0$ (in months) with
\[
l_{10}(a)= \left\{
\begin{split}
& e^{-\mu a}, a\le 2.9,\\
&  e^{-\mu \left(2a-2.9\right)}, a > 2.9
\end{split}
\right. \quad 
l_{20}(a)=  \left\{
\begin{split}
& e^{-\mu a}, a\le 120,\\
&  e^{-\mu \left(2a-120\right)}, a > 120.
\end{split}
\right.
\]
By setting $N_0=S(0)+\|i_{10} \|_{L^1}+ \|i_{20} \|_{L^1}$, the constants $c_{01}$ and $c_{02}$ are scaling coefficients given by $c_{01}= \frac{0.025N_0}{100\|l_{10} \|_{L^1}}  $ and $c_{02}= \frac{0.025N_0}{100\|l_{20} \|_{L^1}} $ such that the HIV epidemic is initialized with a disease prevalence of $ \frac{\|i_{10} \|_{L^1}+ \|i_{20} \|_{L^1}}{N_0} \times 100=0.05\%$.

Numerical simulations of Figure \ref{Figure_model_simu} are based on the reference values of the model parameters defined previously and summarized in Table \ref{tab-model-parameters}. For those parameter values, $R_0=2.55$ and the dynamics of susceptible, $S(t)$, and the total number of HIV stage $j$, $I_j(t)=\int_0^\infty i_j(t,a) \d a$ with $j=1,2,3$, at time $t$ are given in Figure \ref{Figure_model_simu}. Furthermore, the total number of HIV late stage $I_3$ is always be zero in the first 120 months (Figure \ref{Figure_model_simu}, right-bottom). This behaviour is explained by the fact that, in Figure \ref{Figure_model_simu}, the duration of HIV stage 2 is set 120 months and the initial population in HIV late stage is set to zero ($i_{30}(a)\equiv 0$).
\begin{figure}[!htp]
	\center{
		\begin{tabular}{cc}
			\includegraphics[width=2.5in] {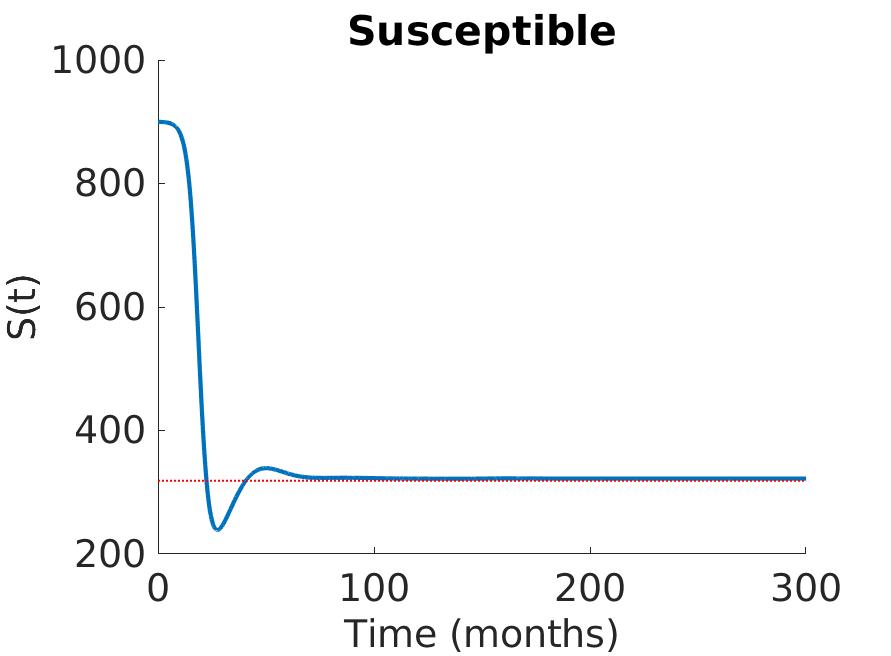} & \includegraphics[width=2.5in] {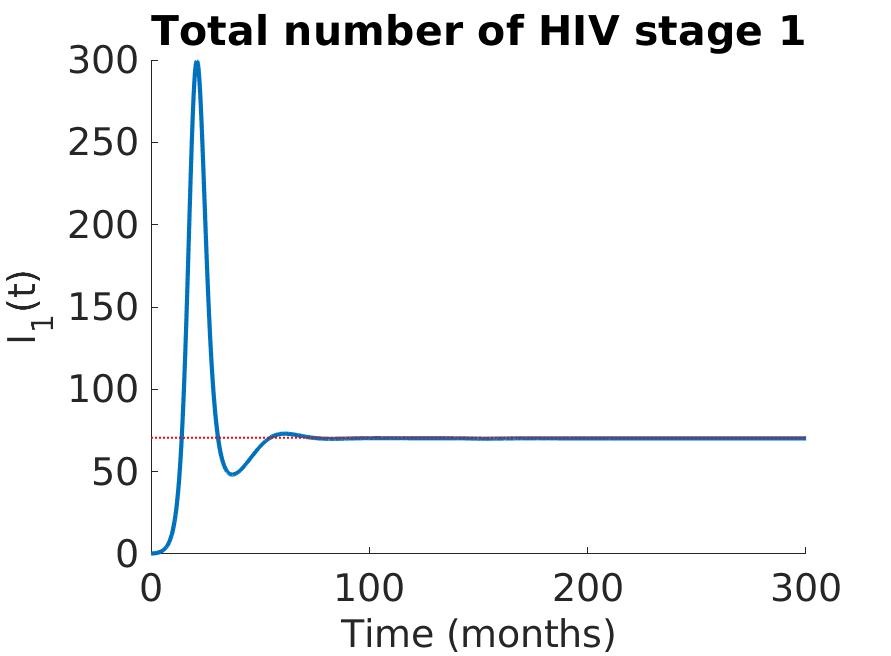} \\
			& \\
			\includegraphics[width=2.5in] {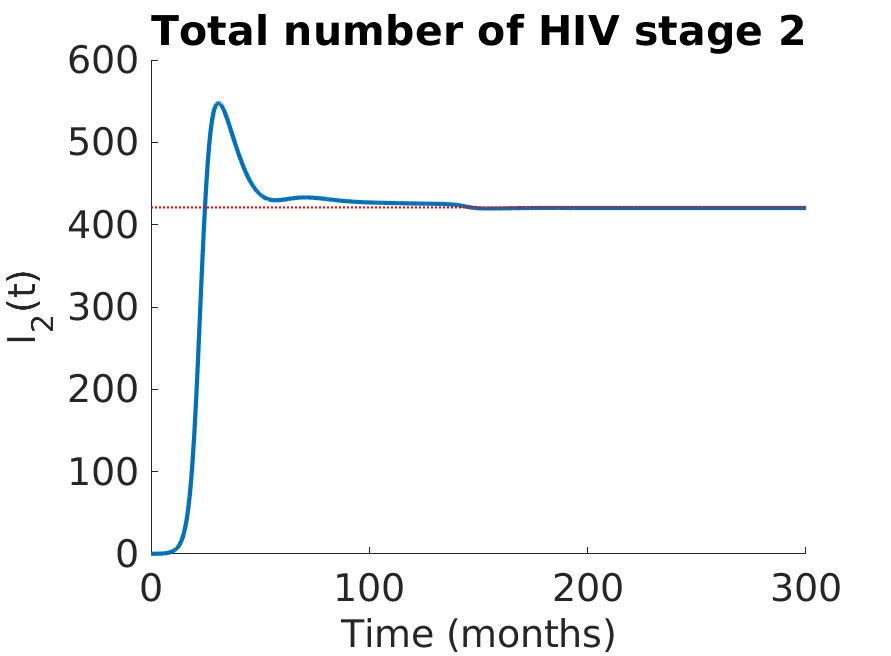} & \includegraphics[width=2.5in] {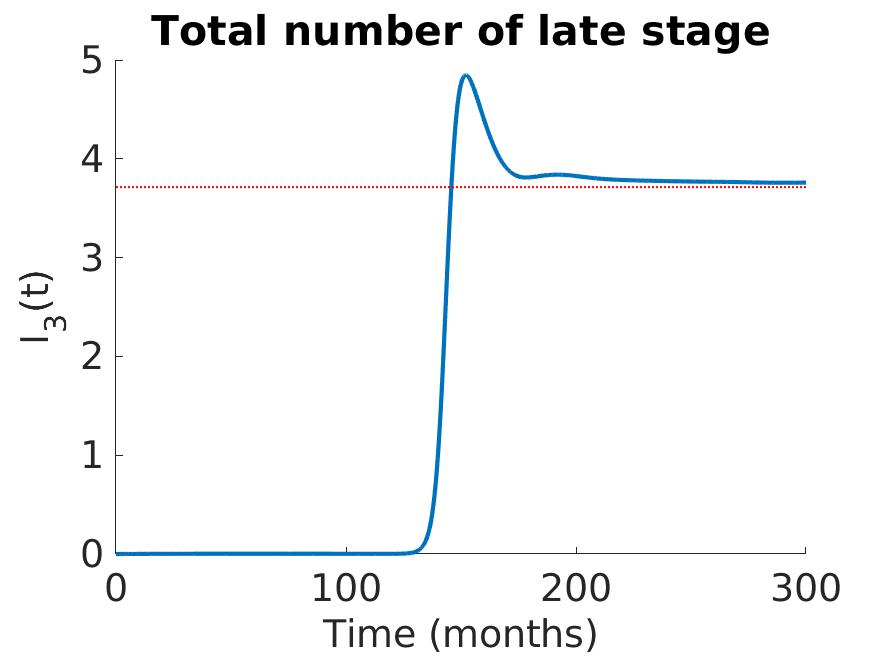}
		\end{tabular}
	}
	\caption{Typical epidemics simulated by the model.  Parameters of the model are set to their reference values given in Table \ref{tab-model-parameters} leading to $R_0=2.55$. The figure illustrates the dynamics of susceptible, $S(t)$, and the total number of HIV stage $j$, $I_j(t)=\int_0^\infty i_j(t,a) \d a$, at time $t$. For each figure, the dot line represents the endemic equilibrium of the model.} \label{Figure_model_simu}
\end{figure}

\subsection{Global sensitivity analysis}
Global sensitivity analyses \cite{Satelli2008} quantify the relative importance of model parameters by partitioning the variance of output variables into those resulting from the main effects of parameters and their higher-order interactions. Here we study the sensitivity of the HIV late stage $I_3^{tot}= \int_0^{T} \int_0^\infty i_3(t,a) \d a\d t$ to the four parameters $\gamma_1(\cdot)$, $\gamma_2(\cdot)$, $\beta_1$ and $\beta_2$. The range of variation accounting for the known biological variability of above parameters is assigned in Table\ref{tab-model-parameters}. Actually, the variability of $\gamma_j(\cdot)$ is determined by the one of $T_0^j$ (see Section \ref{sec-para1}). Sensitivity indices is estimated by fitting an ANOVA (Analysis of variance) linear model, including third-order interactions, to the data generated by simulation. Note that this ANOVA linear model fitted very well with 99\% of variance explained. The model is implement with MatLab software and the ANOVA analysis with R software (http://www.r-project.org/). Sensitivity analyses indicate that the HIV progression rate from stage 2 to 3 ($\gamma_2$) is the most influential factor of the HIV late stage $I_3^{tot}$ (79\% of the variance,
Figure \ref{Figure_sensi_analysis}). The next factor, the transmission rate of stage 1 ($\beta_1$) accounting for 8\% of the variance, is followed by the HIV progression rate from stage 1 to 2 ($\gamma_1$) and the transmission rate of stage 2 ($\beta_2$) ($<$5\% of variance explained for each).
\begin{SCfigure}
	\centering
	\caption{Sensitivity indices of the HIV late stage $I_3^{tot}= \int_0^{T} \int_0^\infty i_3(t,a) \d a\d t$. The black parts of bars correspond to the main indices (effect of the factor alone) and full bars correspond to total indices (white parts corresponds to the effect of the factor in interaction with all other factors).} \label{Figure_sensi_analysis}
	\includegraphics[width=.65\textwidth]%
	{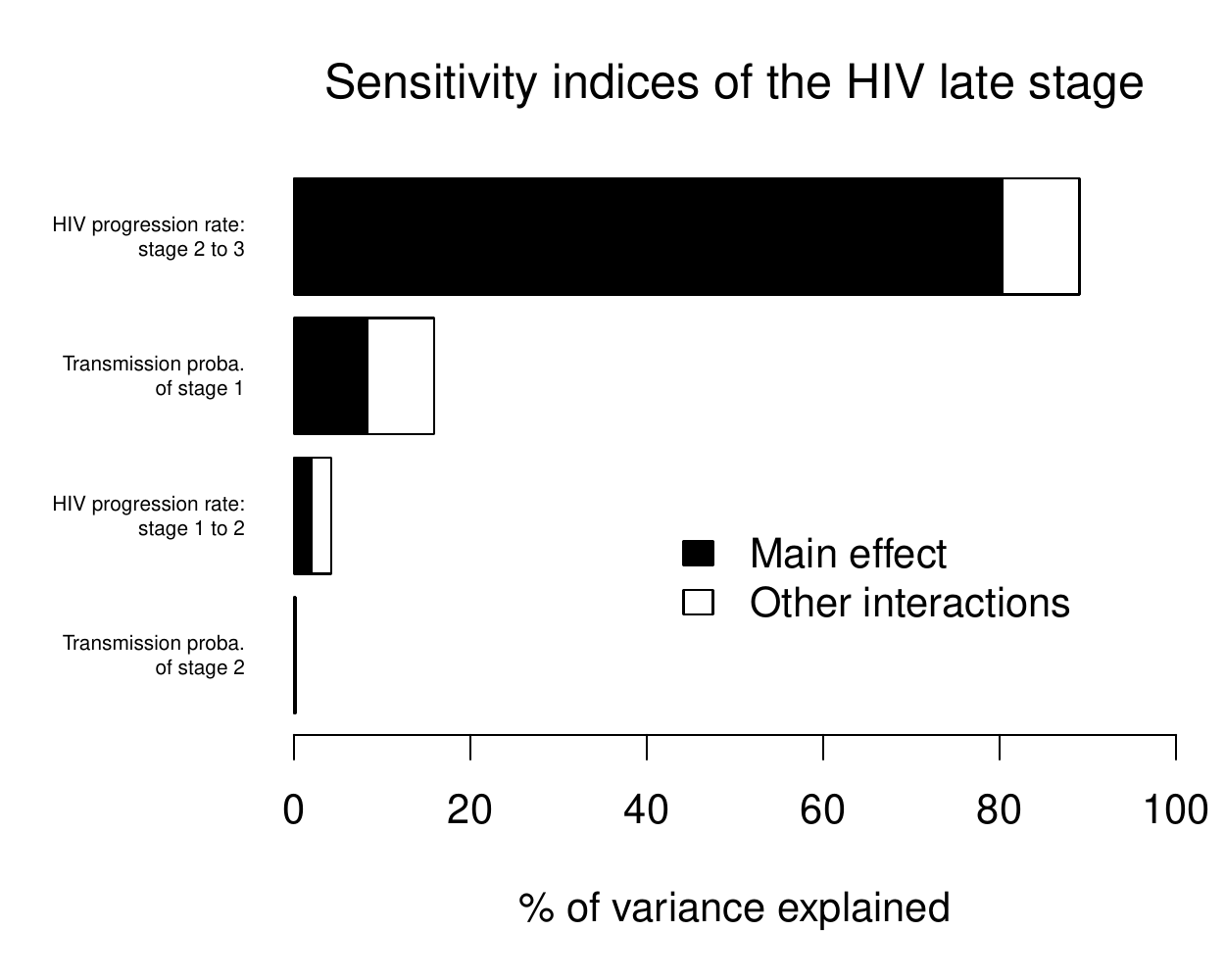}
\end{SCfigure}

\section{Optimal intervention strategies} \label{sec-optimal-model}
Today, more tools are available to prevent HIV such as using condoms the right way every time you have sex or by taking newer HIV prevention medicines such as pre-exposure prophylaxis and post-exposure prophylaxis \cite{center}. But, when living with HIV, up to date, the most important intervention is taking medicines to treat HIV (called ART). Although a cure for HIV does not yet exist, ART can keep healthy for many years if taken consistently and correctly and greatly reduce the chance of transmitting to a partners. Those options are supported by the sensitivity analysis: (i) prevention methods, by the HIV transmission rates and (ii) treatment of HIV, by the HIV progression rates (Figure \ref{Figure_sensi_analysis}). However, only ART is consider as an intervention strategy in this note.

\subsection{Extended model with intervention methods} \label{sec-optimal}
In addition to the previous state variables, $S(t)$-susceptible, $i_1(t,a)$-HIV stage 1 (who are not under ART) and $i_2(t,a)$-HIV stage 2 (who are not under ART), ART interventions induced four additional state variables: $i_{1,TF}(t,a)$-HIV stage 1 and $i_{2,TF}(t,a)$-HIV stage 2 with ART failure or drop out, as well as $i_{1,TS}(t,a)$-HIV stage 1 and $i_{2,TS}(t,a)$-HIV stage 2 with ART success. We also introduce the probability of treatment drop out: $p_1$; at HIV stage 1, $p_2$, at HIV stage 2 (with no ART at HIV stage 1) and $p_{2,TF}$, at HIV stage 2 (with ART drop out at HIV stage 1). This differential infectivity is supported by the fact that individuals who have dropped out of treatment progress through subsequent HIV stages twice as fast as treatment-naïve individuals \cite{Eaton2014}. Therefore, people in class $i_{1,TF}$ (resp. $i_{2,TF}$) progress at rate $\gamma_{1,TF}(a)$  (resp. $\gamma_{2,TF}(a)$) to the HIV  stage 2 (resp. late stage). Then, the force of infection and overall progression into subsequent stages write $
E_1(t)= \int_0^\infty \beta(a) \left[i_1+i_{1,TF} +\varepsilon( i_2+i_{2,TF}) +\delta i_3\right] (t,a) \d a$, $
E_2(t)= \int_0^\infty \gamma_1(a) i_1(t,a) \d a,$ $
E_{2,TF}(t)= \int_0^\infty \gamma_{1,TF}(a) i_{1,TF}(t,a) \d a,$ and  $E_3(t)= \int_0^\infty \left(\gamma_2(a) i_2(t,a) +\gamma_{2,TF}(a) i_{2,TF}(t,a)\right) \d a.$ The total population is set to $P(t)=S(t)+ \sum_{j=1}^3 \int_{0}^\infty i_j(t,a)\d a+ \sum_{j=1}^2 \int_{0}^\infty i_{j,TF}(t,a)\d a+ \sum_{j=1}^2 \int_{0}^\infty i_{j,TS}(t,a)\d a $. The transfer diagram of the model becomes:
\begin{figure}[H]
\includegraphics[width=1\textwidth]%
	{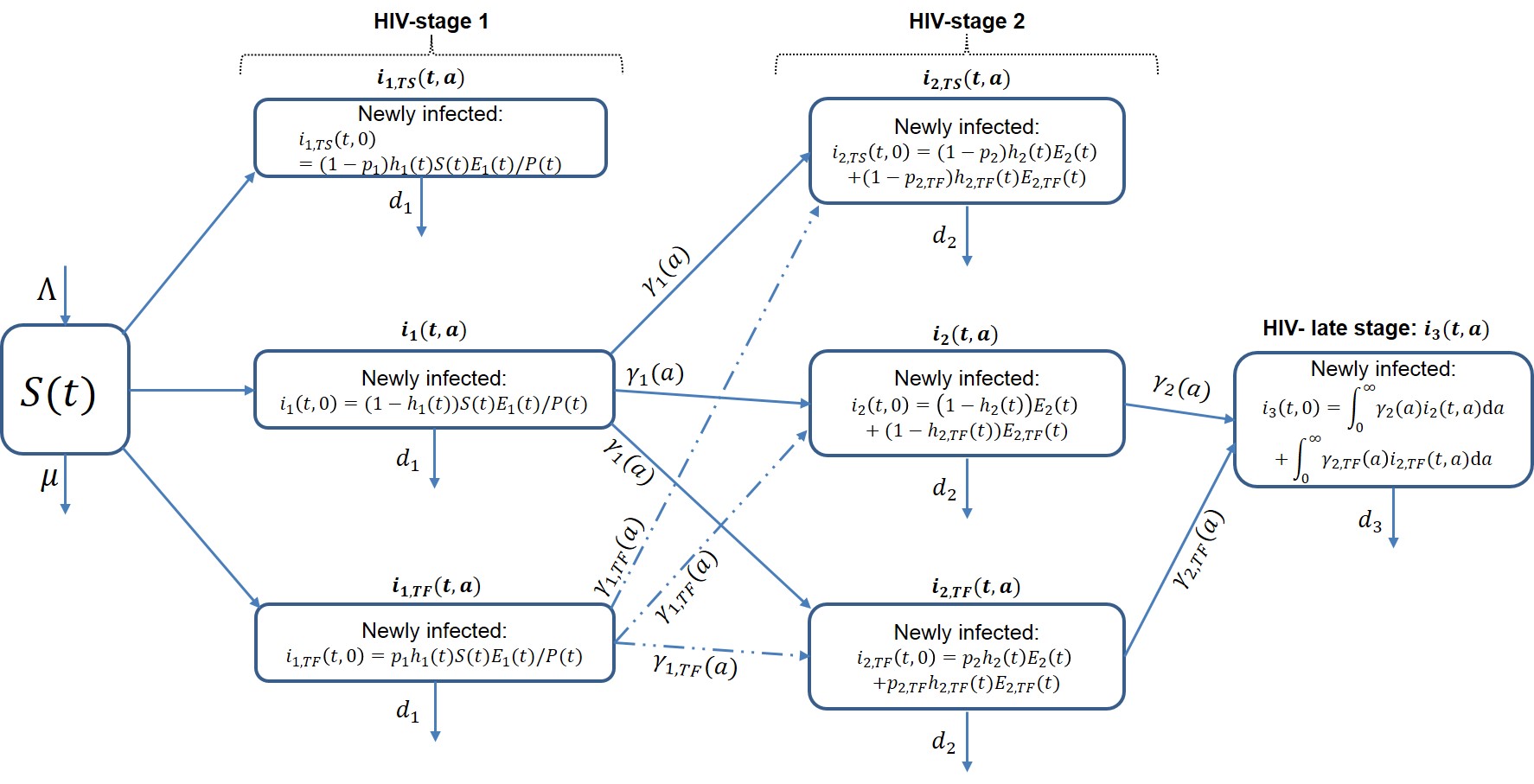}
\end{figure}\vspace{-.8cm}
Three interventions strategies, called controls, are include into our initial model \eqref{model1}. Controls are represented as functions of time and assigned reasonable upper and lower bounds. At a time $t$, we define the proportion of individuals: (i) $h_1(t)$;  under ART at HIV stage 1, (ii) $h_2(t)$, under ART at HIV stage 2 (which were not under ART at HIV stage 1) and (iii) $h_{2,TF}(t)$, under ART at HIV stage 2 (with ART drop out at HIV stage 1). Then, the system describing our model with controls writes
\begin{equation}
	\begin{cases}
	\begin{split}
	 \dot S(t)=&\Lambda-\mu S(t)-S(t) E_1(t)/P(t), \\
	\left(\partial_t+\partial_a\right) i_1(t,a)=& - \left( \gamma_1(a) + d_1(a) \right) i_1(t,a),\\
	\left(\partial_t+\partial_a\right) i_{1,TF}(t,a)=& - \left( \gamma_{1,TF}(a) + d_1(a) \right) i_{1,TF}(t,a),\\
	\left(\partial_t+\partial_a\right) i_{1,TS}(t,a)=& - d_1(a)  i_{1,TS}(t,a),\\
	\left(\partial_t+\partial_a\right) i_2(t,a)=& - \left( \gamma_{2}(a) + d_2(a) \right) i_2(t,a),\\
	\left(\partial_t+\partial_a\right) i_{2,TF}(t,a)=& - \left( \gamma_{2,TF}(a) +d_2(a)\right) i_{2,TF}(t,a),\\
	\left(\partial_t+\partial_a\right) i_{2,TS}(t,a)=& -  d_2(a) i_{2,TS}(t,a),\\
	\left(\partial_t+\partial_a\right) i_3(t,a)=& -  d_3(a) i_3(t,a),
	\end{split}
	\end{cases}\label{model2}
	\end{equation}
with the boundary conditions
\begin{equation}
\begin{cases}
\begin{split}
i_1(t,0)=& \left(1- h_1(t)\right) S(t) E_1(t)/P(t),\\
i_{1,TF}(t,0)=& p_1 h_1(t) S(t) E_1(t)/P(t),\\
i_{1,TS}(t,0)=& (1-p_1) h_1(t) S(t) E_1(t)/P(t),\\
i_2(t,0)=& \left(1- h_2(t) \right) E_2(t) + \left(1- h_{2,TF}(t) \right) E_{2,TF}(t), \\ i_{2,TF}(t,0)=& p_2 h_2(t)  E_{2}(t)+ p_{2,TF} h_{2,TF}(t) E_{2,TF}(t),\\
i_{2,TS}(t,0)=& (1-p_2) h_2(t)  E_{2}(t) + (1- p_{2,TF}) h_{2,TF}(t) E_{2,TF}(t),\\
i_3(t,0)=& E_{3}(t).
\end{split}
\end{cases}\label{model2-cond}
\end{equation}
We set $S^h(t)$, and  $y^h(t,a)=$\\ $ \left( i_1^h(t,a), i_{1,TF}^h(t,a), i_2^h(t,a), i_{2,TF}^h(t,a), i_3^h(t,a),i_{1,TS}^h(t,a),i_{2,TS}^h(t,a) \right)^T$ the solution of \eqref{model2}-\eqref{model2-cond} associated to the control scheme $h= \left(h_1,h_2,h_{2,TF}\right)^T$. To illustrate the dependency with respect to $h$, we also set $P^h(t)=P(t)$, $E_1^h(t)= E_1(t)$, $E_2^h(t)= E_2(t)$, $E_{2,TF}^h(t)= E_{2,TF}(t)$, and $E_3^h(t)= E_3(t)$. Then, problem \eqref{model2}-\eqref{model2-cond} rewrites 
\begin{equation}\label{model2-compact}
\begin{cases}
\dot S^h(t)= g_S \left(S^h(t),E_1^h(t)/P^h(t)\right),\\
\left(\partial_t+\partial_a\right)y^h(t,a)= f(a) y^h(t,a),\\
y^h(t,0)= \phi \left(h(t),E_1^h(t)/P^h(t),E_2^h(t), E_{2,TF}^h(t),E_3^h(t),E_{3,TF}^h(t) \right),
\end{cases}
\end{equation}
wherein $g_S$ is given by the right-hand side of \eqref{model2} for the $S$-compartment; $f(a) y^h$ is the linear operator given by the the right-hand side of \eqref{model2} for the $\left( i_1, i_{1,TF}, i_2, i_{2,TF},i_3,i_{1,TS},i_{2,TS} \right)$-compartment and $\phi$ is given by the the right-hand side of \eqref{model2-cond} for the $\left( i_1, i_{1,TF}, i_2, i_{2,TF},i_3,i_{1,TS},i_{2,TS} \right)$-compartment.

\subsection{Optimal control problem}
We assume that a successful scheme is one which reduces the progression to the AIDS stage. Therefore, the control scheme is optimal if it minimizes the objective functional

\begin{equation*}\label{5}
\begin{array}{ll}
J(h)=&\int_{0}^{T_{f}}  \int_0^\infty B(a) \left[\gamma_2(a) i_2(t,a) + \gamma_{2,TF}(a) i_{2,TF}(t,a) \right] \d a \d t   \\&+ \int_0^{T_f} [C_1 h_1^{2}(t)+ C_2 h_2^2(t)+ C_3 h_{2,TF}^2(t) ]\d t
\end{array}
\end{equation*}
where $B$ and $C_j$ are balancing coefficients transforming the integral into cost expended over a finite period of $T_f$ months, see Table \ref{tab-model-parameters}. The first integral, multiply by $B$, is the economic losses from individuals progressing into AIDS stage and the second integral represents the costs for the implementation of three controls. Quadratic expressions of controls are included to indicate non-linear costs potentially arising at high treatment levels.

Our aim is to find $h^*$ satisfying 
\begin{equation} \label{PB}
J(h^*)=\min_{h \in \mathcal{U}} J(h),
\end{equation}
on the set 

\begin{eqnarray*} 
	\mathcal{U}=\left\{
	\begin{array} {l}
		h \in L^\infty(0,T_{f}): 0\le h_1( \cdot) \le h_1^{\max}; 0\le h_2( \cdot) \le h_2^{\max};\\ 0\le h_{2,TF}( \cdot) \le h_{2,TF}^{\max}
	\end{array}\right\},\end{eqnarray*}
where $h_j^{\max}\le 1$ are positive measurable functions.

\subsection{The necessary optimality condition} \label{sec-optimality-condition}
In order to deal with the necessary optimality conditions, we use some results in Feichtinger {\it et al}. \cite{Feichtinger2003}. 
For a given solution $S(t)$,  $y(t,a)=$\\$ \left( i_1, i_{1,TF}, i_2, i_{2,TF},i_3,i_{1,TS},i_{2,TS} \right)(t,a)$  and $\left(E_1(t)/P(t),E_2(t),E_{2,TF}(t),E_3(t)\right)$ of \eqref{model2}-\eqref{model2-cond}, we introduce the following adjoint functions $\lambda_S(t)$,
$\xi(t, a)=(\lambda_{i_1}, \lambda_{i_1,TF}, \lambda_{i_2}, \lambda_{i_{2,TF}}, \lambda_{i_3}, \lambda_{i_{1,TS}},\lambda_{i_{2,TS}})(t,a)$ and $ \left( \zeta_1(t), \zeta_2(t), \zeta_{2,TF}(t), \zeta_3(t) \right)$.

%

The following system holds from \cite{Feichtinger2003}
\begin{equation}\label{adj}
\left\{\begin{array}{rll}
\displaystyle\dot{\lambda}_{S}(t)&=& \big[E_1(t)/P(t)+\mu\big] \lambda_S(t),\\
(\partial_t+\partial_a)\lambda_{i_1}(t,a)&=& \left[\gamma_1(a)+ d_1(a) \right] \lambda_{i_1}(t,a) -\beta(a)\zeta_1(t)/P(t)- \gamma_1(a)\zeta_2(t), \\
(\partial_t+\partial_a)\lambda_{i_{1,TF}}(t,a)&=& \left[\gamma_{1,TF}(a)+ d_1(a) \right] \lambda_{i_{1,TF}}(t,a) -\beta(a)\zeta_1(t)/P(t)- \gamma_{1,TF}(a)\zeta_{2,TF}(t),\\
(\partial_t+\partial_a)\lambda_{i_2}(t,a)&=& -B \gamma_2(a) + \left[\gamma_2(a) + d_2(a) \right] \lambda_{i_2}(t,a) - \varepsilon \beta(a) \zeta_1(t)/P(t) -\gamma_2(a)\zeta_3(t),\\
(\partial_t+\partial_a)\lambda_{i_{2,TF}}(t,a)&=& -B \gamma_{2,TF}(a) + \left[\gamma_{2,TF}(a) + d_2(a) \right] \lambda_{i_{2,TF}}(t,a)\\
&& - \varepsilon \beta(a) \zeta_1(t)/P(t)- \gamma_{2,TF}(a)\zeta_3(t),\\
(\partial_t+\partial_a)\lambda_{i_3}(t,a)&=&    d_3(a) \lambda_{i_3}(t,a) - \delta \beta(a) \zeta_1(t)/P(t) ,\\
(\partial_t+\partial_a)\lambda_{i_{1,TS}}(t,a)&=&    d_1(a) \lambda_{i_{1,TS}} ,\\
(\partial_t+\partial_a)\lambda_{i_{2,TS}}(t,a)&=&    d_2(a) \lambda_{i_{2,TS}} ,\\
\zeta_1(t)&=& (1- h_1(t)) S(t) \lambda_{i_1}(t,0) + p_1 h_1(t) S(t) \lambda_{i_{1,TF}}(t,0)\\
&& + (1-p_1) h_1(t) S(t) \lambda_{i_{1,TS}}(t,0),\\
\zeta_2(t)&=& (1- h_2(t))  \lambda_{i_2}(t,0) + p_2 h_2(t)  \lambda_{i_{2,TF}}(t,0) + (1-p_2) h_2(t) \lambda_{i_{2,TS}}(t,0),\\
\zeta_{2,TF}(t)&=& (1- h_{2,TF}(t))  \lambda_{i_2}(t,0) + p_{2,TF} h_{2,TF}(t)  \lambda_{i_{2,TF}}(t,0)\\
&& + (1-p_{2,TF}) h_{2,TF}(t) \lambda_{i_{2,TS}}(t,0),\\
\zeta_3(t)&=&  \lambda_{i_3}(t,0),\\
\end{array}\right.
\end{equation}
with the boundary conditions
\begin{equation}\label{adj-cond}
\lambda_{S}(T_f)=0, \quad
\lambda_{v}(T_f,a)=0,  \text{ for } v\in \{i_1,i_{1,TF},i_2,i_{2,TF},i_3,i_{1,TS},i_{2,TS}\} \text{ and for all } a> 0.
\end{equation}

The Hamiltonian of System \eqref{model2}-\eqref{model2-cond} is given by 
$$H(t,h)= \xi(t,0)\cdot \phi\left(h,E_1/P,E_2,E_{2,TF},E_3\right) + \lambda_S(t) g_S(t) +L_0(t) + \int_0^\infty L(t, a) \d a ,$$ 
where $L$ and $L_0$ are the first and second integrand of $J$ respectively.

Moreover, if $h^*$ is a solution of \eqref{PB}, then it is characterized by
\begin{equation}\label{control-carac}
\begin{split}
 h_1^\ast(t)=&\max(0, \min(\hat{h}_1(t), h_{1}^{\max} )),\\
h_{2}^\ast(t)= & \max(0, \min(\hat{h}_{2}(t), h_{2}^{\max} )),\\
 h_{2,TF}^\ast(t)=& \max(0, \min(\hat{h}_{2,TF}(t), h_{2,TF}^{\max} )),
\end{split}
\end{equation}
wherein 
\begin{equation*}
\begin{array}{rll}
\hat{h}_1(t)&=&\displaystyle\frac{\left(\lambda_{i_1}(t,0)- p_1 \lambda_{i_{1,TF}}(t,0) -(1-p_1) \lambda_{i_{1,TS}}(t,0) \right) S(t) E_1(t)/P(t)}{2C_1}, \\
\hat{h}_2(t)&=&\displaystyle\frac{\left(\lambda_{i_2}(t,0)- p_2 \lambda_{i_{2,TF}}(t,0)-(1-p_2)\lambda_{i_{2,TS}}(t,0) \right)  E_2(t)}{2C_2} ,\\
\hat{h}_{2,TF}(t)&=&\displaystyle\frac{\left(\lambda_{i_2}(t,0)- p_{2,TF} \lambda_{i_{2,TF}}(t,0)- (1-p_{2,TF}) \lambda_{i_{2,TS}}(t,0) \right)  E_{2,TF} (t)}{2C_3}.
\end{array}
\end{equation*}

The value function $\hat{h}_1$ is obtained from equation $\displaystyle\frac{\partial H}{\partial h_1}=0$, using the boundary condition whenever $0\le h_1^\ast(t) \le h_1^{\max}$, and similarly for $\hat{h}_{2}$ and $\hat h_{2,TF}$. The proof of the existence of such controls is given in Section \ref{sec-control-existence}.

The state system \eqref{model2}-\eqref{model2-cond} and the adjoint system \eqref{adj}-\eqref{adj-cond} together with the control characterization \eqref{control-carac} form the optimality system to be solved numerically. Since the state equations have initial conditions and the adjoint equations have final time conditions, we cannot solve the optimality system directly by only sweeping forward in time. Thus, an iterative algorithm, "forward-backward sweep method", is used (see \cite{Lenhart2007}).

\section{Discussion} \label{sec-discussion}
For all simulations consider in this section, in addition to the parameter reference values given by Table \ref{tab-model-parameters}, the number of months for the strategy deployment is $T_f= 420$ (35 years). Moreover, the progression rate in classes $i_{1,TF}$ and $i_{2,TF}$, $\gamma_{1,TF}(a)$ and $\gamma_{2,TF}(a)$, have the same shape as $\gamma_{1}(a)$ and $\gamma_{2}(a)$ define in Section \ref{sec-para1}. However, the duration of each stage $i_{1,TF}$ and $i_{2,TF}$ is $T_0^1/2$ and $T_0^2/2$ respectively \cite{Eaton2014}. At the beginning of intervention strategies, initial condition of state variables $i_1$, $i_2$ and $i_3$ is the same as in Section \ref{sec-initial1}. Furthermore, we set $i_{1,TS}(0,a)= i_{1,TF}(0,a)= i_{2,TS}(0,a) = i_{2,TF}(0,a)=0$ for all $a\ge 0$.

\paragraph{Performance of controls strategies.} 
The performance is estimated by assessing the total number of AIDS cases during $T$ months relatively to AIDS cases without any strategy. Formally, the performance of the intervention strategy $h=(h_1,h_2,h_{2,TF})$ is 
\[
\Delta^h(P)= \frac{\int_0^{T} \int_0^\infty i_3(t,a)\big |_h \d a \d t } { \int_0^{T} \int_0^\infty i_3(t,a)\big |_{h= 0} \d a \d t },
\]
wherein $P= \left(p_1,p_2,p_{2,TF}\right)$. $\Delta^h(P)$ provides an estimate of the number of AIDS obtained with the control $h$ ($\int \int i_3(t,a)\big |_h \d a \d t$) relatively to the number of AIDS that would have been obtained without any control ($\int \int i_3(t,a)\big |_{h= 0} \d a \d t$). For example, $\Delta^h(P)=0.3$ indicates that the number of AIDS with control is $1/0.3 \approx 3$ times smaller than without control. However, a value of $\Delta^h(P)>1$ indicates that intervention strategy has negative impact on the epidemics outbreak.

Figure \ref{Figure_model_control1} illustrates the effect of intervention strategies on the number of AIDS, $I_3^{\text{tot}}(t)= \int_0^\infty i_3(t,a) \d a $. Parameters of the model are set to their reference values given in Table \ref{tab-model-parameters} leading to $R_0=2.55$ (without controls) and we also set $p_1=p_2=p_{2,TF}= 10\%$. Optimal three-part intervention strategies, $h_1$, $h_2$ and $h_{2,TF}$, provide considerable reductions in the severity of the projected outbreaks (Figure \ref{Figure_model_control1}). Indeed, the performance of controls is $\Delta^h(P) \approx 0.018$, meaning that the total number of AIDS cases with controls is 55 times smaller than without controls.

\begin{figure}[!htp]
	\center{
		\begin{tabular}{cc}
			\includegraphics[width=2.5in] {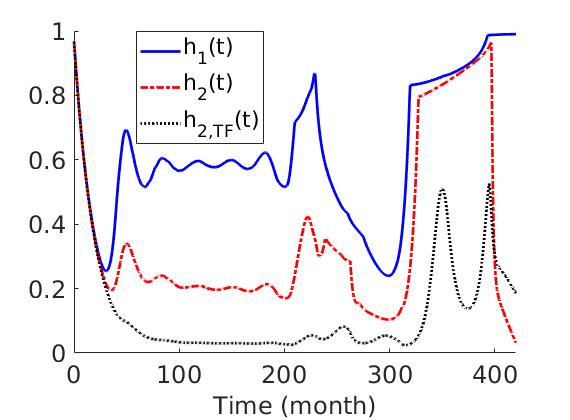} & \includegraphics[width=2.5in] {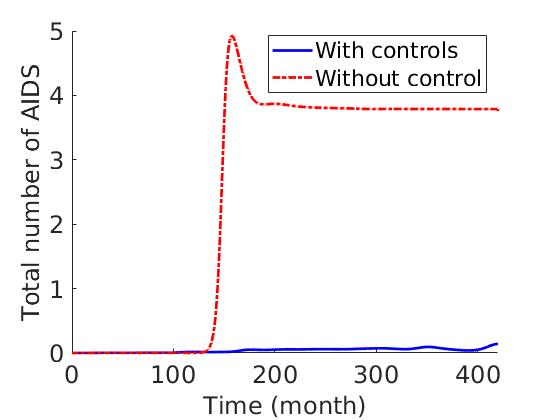} 
		\end{tabular}
	}
	\caption{Effect of intervention strategies on the number of AIDS, $I_3^{\text{tot}}(t)= \int_0^\infty i_j(t,a) \d a $. (Left) Intervention strategies $h_1$, $h_2$ and $h_{2,TF}$ with respect to the time $t$. (Right) The dynamics of the number of AIDS cases with and without controls. Parameters of the model are set to their reference values given in Table \ref{tab-model-parameters} leading to $R_0=2.55$ (without controls). We also set $p_1=p_2=p_{2,TF}= 10\%$.} \label{Figure_model_control1}
\end{figure}  

\paragraph{Effect of intervention only at HIV stage 1 or 2.} Notice that the optimal control problem can be formulated to find the optimal strategy of each HIV stage intervention method when used alone. Moreover, without any intervention at HIV stage 1, the only plausible intervention at stage 2 is $h_2$. The control $h_1$ alone has a significant effect on the epidemic outbreaks with a performance $\Delta^{h_1}(p_1) \le 0.2$ when the probability of treatment drop out at HIV stage 1 $p_1$ ranges from 0 to 80\% in the host population. In other words, for $p_1\le 0.8$, the total number of AIDS cases with controls $h_1$ is, at least, 5 times smaller than without controls. See Figure \ref{Figure_perform_controls_alone}. The control $h_2$ alone has a significant effect on the epidemic outbreaks only when the probability of treatment drop out at HIV stage 2, $p_2$ (with not ART at stage 1) is very small ($p_2< 0.1$). The performance $\Delta^{h_2}(p_2)$ increases linearly with the value of $p_2$ and crosses the unity around $p_2=10\%$, after what the intervention strategy will have a negative effect on the epidemic outbreaks in the host population, Figure \ref{Figure_perform_controls_alone}. Therefore, before introducing ART, investigations should be addressed to know whether the host population is well sensibilized on ART treatment or not. These investigations will probably help in reducing the probability of treatment drop out.

\begin{SCfigure}
	\caption{Performance of HIV intervention strategy when used alone. (Dotted line) Performance of HIV stage 1 intervention alone $\Delta^{h_1}(p_1)$ with respect to the probability of ART drop out at stage 1, $p_1$. (Solid line) Performance of HIV stage 2 intervention alone $\Delta^{h_2}(p_2)$ with respect to the probability of ART drop out at stage 2, $p_2$ (with no ART at stage 1).} \label{Figure_perform_controls_alone}
	\includegraphics[width=3in] {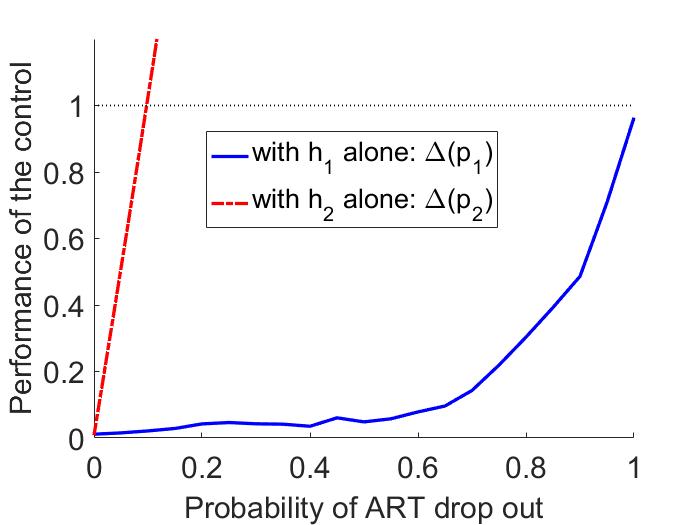} 
\end{SCfigure}

\paragraph{Combined effect of intervention at both HIV stages.} Controls $h_1$ and $h_{2,TF}$ have a significant effect on the disease outbreak with the performance of both controls $\Delta^{h_1,h_{2,TF}} \left(p_1,p_{2,TF}\right) < .5$, when the probability of treatment drop out at HIV stage 1, $p_1 \le 60\%$ and what ever the probability of treatment drop out at HIV stage 2 (with ART failure at stage 1), $p_{2,TF} \in (0,1)$. See Figure \ref{Figure_perform_h1_h2}(Left). However, for high values of $p_1$ ($p_1>.6$) and for $p_{2,TF}>.2$, both controls have a negative effect on the epidemics outbreak with the performance  $\Delta^{h_1,h_{2,TF}} \left(p_1,p_{2,TF}\right) \ge 1$, see Figure \ref{Figure_perform_h1_h2}(Left). These configurations are quite similar with the combined effect of controls  $h_1$ and $h_{2}$, see Figure \ref{Figure_perform_h1_h2}(Right). However, notice that controls $h_1$ and $h_{2}$ performed better than controls $h_1$ and $h_{2,TF}$, even for values of $p_1$ up to $80\%$. 

\begin{figure}[!htp]
	\center{
		\begin{tabular}{cc}\includegraphics[width=3.1in] {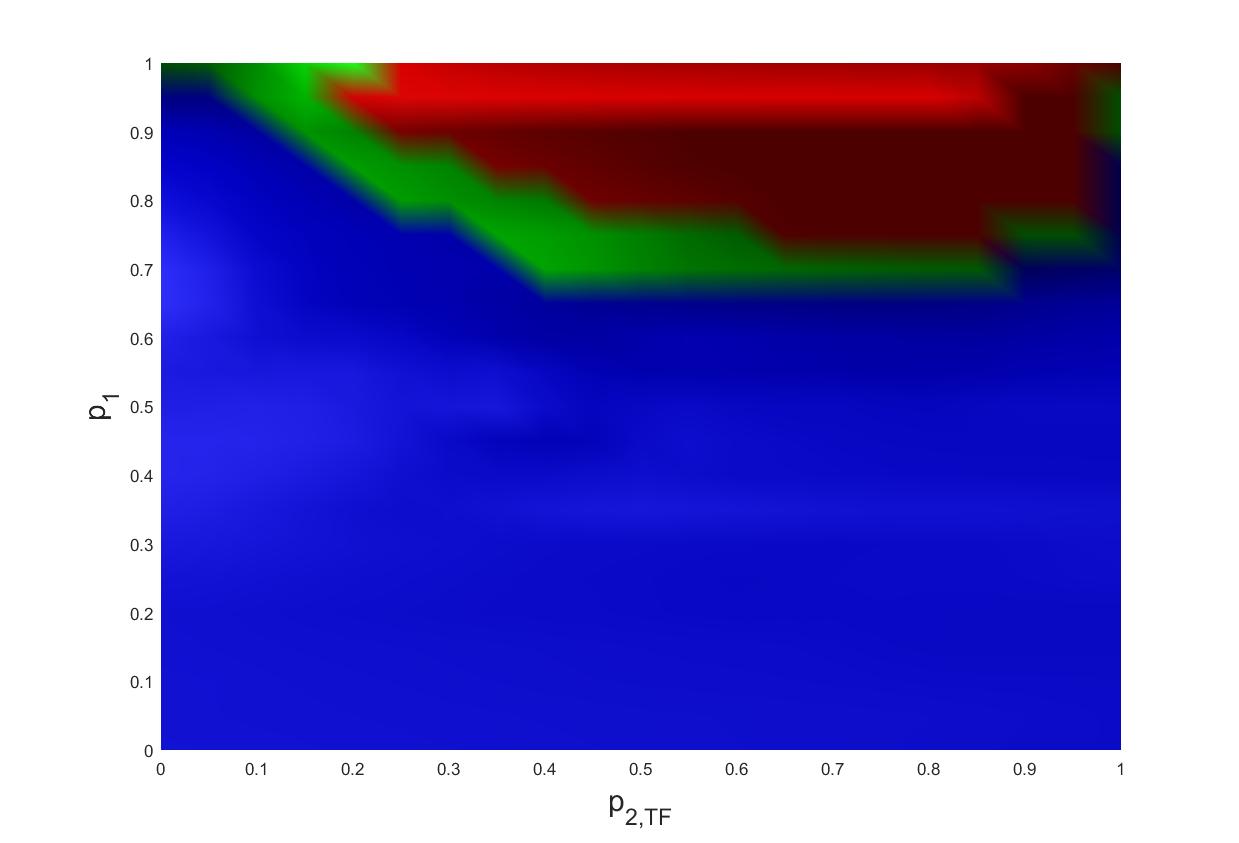} & \includegraphics[width=3in] {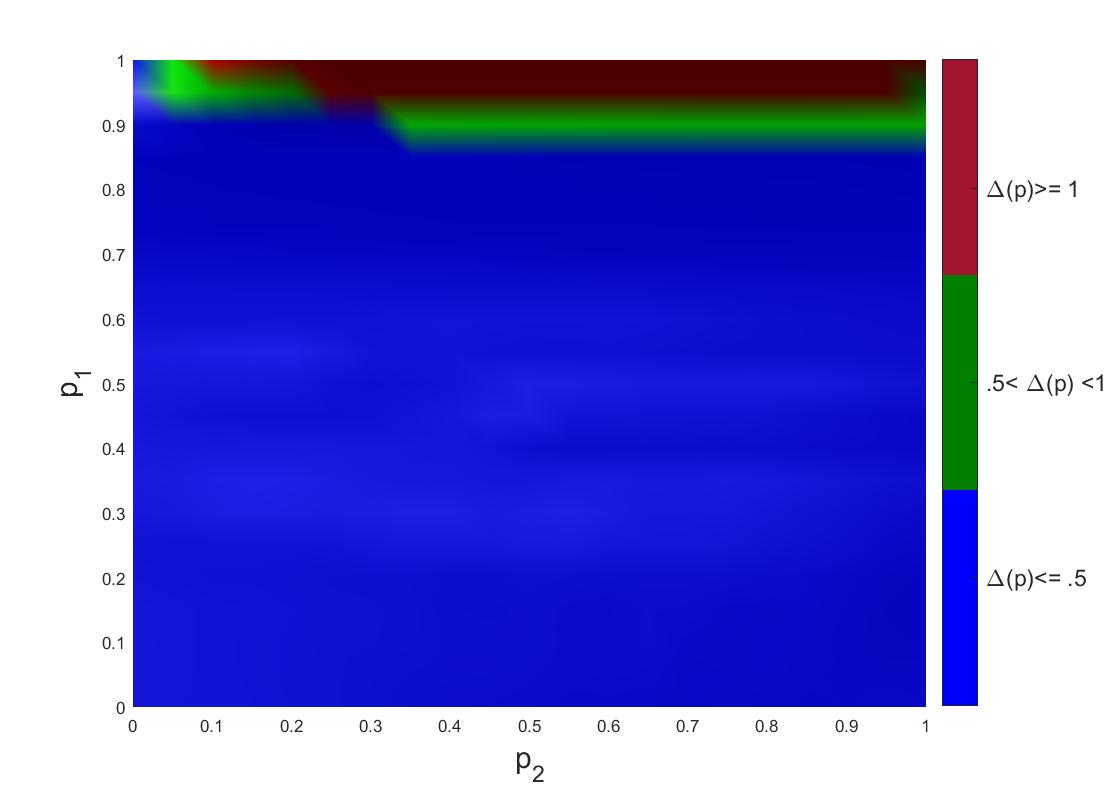}
		\end{tabular} 
	}
	\caption{Combined effect of intervention at HIV stages 1 and 2. (Left) The performance $\Delta^{h_1,h_{2,TF}} \left(p_1,p_{2,TF}\right)$ of controls $h_1$ and $h_{2,TF}$ with respect to probabilities of treatment drop out at HIV stage 1, $p_1$ and treatment drop out at HIV stage 2 (with ART failure at stage 1), $p_{2,TF}$. (Right) The performance $\Delta^{h_1,h_2} \left(p_1,p_{2}\right)$ of controls $h_1$ and $h_{2}$ with respect to probabilities of treatment drop out at HIV stage 1, $p_1$ and treatment drop out at HIV stage 2 (with no ART at stage 1), $p_{2}$. (Color figure online)} \label{Figure_perform_h1_h2}
\end{figure}

\begin{small}
	\begin{table}[H]
		\caption{Description of the state variables and parameters of the model}
		\label{tab-model-parameters}      
		\begin{tabular}{lll}
			\hline
			\multicolumn{3}{c}{State variables}\\
			\hline
			$S(t)$& \multicolumn{2}{l}{Susceptible individuals at time $t$} \\
			$i_j(t,a)$& \multicolumn{2}{l}{HIV individuals stage $j$ at time $t$, in the stage for duration $a$ (with no ART)} \\
			$i_{j,TF}(t,a)$& \multicolumn{2}{l}{HIV individuals stage $j$ with ART drop out at time $t$, in the stage for duration $a$} \\
			$i_{j,TS}(t,a)$& \multicolumn{2}{l}{HIV individuals stage $j$ with ART success at time $t$, in the stage for duration $a$} \\
			\hline
			\multicolumn{3}{c}{Fixed model parameters}\\
			Parameters & Description; Value& Ref  \\ 
			\hline
			$\Lambda$ & Entering flux into $S$; 30  & Ass  \\
			$\mu$ & Exit rate of $S$; 1/30 PMs & Ass \\
			$\bar{\gamma}_j$ &  Progression rate from stage $j$ to $j+1$; 1 PMs  & Ass\\
			$d_j$ & Death rate of  HIV stage $j=1,2$; $\mu$ PMs  & Ass \\
			$d_3$ & Death rate of HIV late stage; 0.14 per 1000 PYs+ $\mu$ & \cite{onusida}\\
			$\rho_0$ & The rate of infectiousness; $2.48$ & Ass
			\\ \hline
			\multicolumn{3}{c}{Variable model parameters}\\
			Parameters & Description; Reference value; \{Range\}& Ref  \\ \hline
			$T_0^1$ & Duration of HIV stage 1; 2.90; \{1.23-6.00 \} months & \cite{Hollingsworth2008}\\
			$T_0^2$ & Duration of HIV stage 2; 120; \{108-180 \} months & \cite{Hollingsworth2008}\\
			$\beta_1$ & Transmission hazard stage 1; 276; \{131-509\} per 100 PYs & \cite{Hollingsworth2008}\\
			$\beta_2$ & Transmission hazard stage 2; 10.6; \{7.61-13.3\} per 100 PYs & \cite{Hollingsworth2008}\\ 
			$\beta_3$ & Transmission hazard stage 3; 0 per 100 PYs & Ass\\
			$p_1$& Proba. of ART drop out at HIV stage 1;\{0-1\} &Ass\\
			$p_2$& Proba. of ART drop out at stage 2 (with no ART at stage 1); \{0-1\} &Ass\\
			$p_{2,TF}$& Proba. of ART drop out at stage 2 (with ART failure at stage 1); \{0-1\} &Ass\\
			\hline
			\multicolumn{3}{c}{Cost coefficients in objective functional}\\ \hline 
			B &Balancing coefficient; \{50-80\} \euro &\cite{taverne2008}\\
			$C_j$& Balancing coefficients; \{50-80\} \euro&\cite{taverne2008}\\
			\hline
		\end{tabular}
		PMs=person-months; PYs=person-years; Ass=Assumed.
	\end{table}
\end{small}
\section{Proof of Theorem \ref{thm-semiflow}} \label{sec-proof-semiflow}

It is easy to check that operator $A$ is a Hille-Yosida operator. Then standard results apply to provide the existence and uniqueness of a mild solution to \eqref{model1} (we refer to \cite{Magal2009,Thieme1990,Thieme2011} for more details). The Volterra formulation is also standard and we refer to \cite{Iannelli1994,Webb1985} for more details.

For estimate 2., let $\varphi_0 \in \mathcal{X}_{0+}$; then adding up the $S$ equation
together with the $i_j$ equations of \eqref{model1} yields
\[
\dot P(t)= \frac{d}{dt} \left(S(t)+ \sum_{j=1}^3 \int_0^\infty i_j(t,a)\d a \right) \le \Lambda -\mu P(t).
\]
From where one deduces estimate 2. 

The bounded dissipativity of the semiflow $\left\{ \Phi(t,\cdot) \right\}_t$ is a direct consequence of estimate 2. It remains to prove the asymptotic smoothness. For that ends, let $B$ be a forward invariant bounded subset of $\mathcal{X}_{0+}$. According to the results in \cite{Sell2002} it is sufficient to show that the semiflow is asymptotically compact on $B$. We first claim that 
\begin{claim}\label{claim_lips}
Let Assumption \ref{asym1} be satisfied. Then, functions $E_j$'s are Lipschitz continuous on $\R_+$.
\end{claim}
Therefore, let us consider a sequence of solutions $\left(S^n,i_1^n,i_2^n,i_3^n\right)_p$ that is equibounded in $\mathcal{X}_{0+}$ and a sequence $\{t_n\}_p$ such that $t_n \to +\infty$. Let $I_j^n(t) = \int_{0}^\infty i_j^n(t,\sigma)\d \sigma$, $P^n(t)= S^n(t) + \sum_{j=1}^3 I_j^n(t)$ and  $E_j^n$ (defined by \eqref{eq-E1-E2}); with $j=1,2,3$; the corresponding sequences. Since $S^n$, $P^n$, $I_j^n$'s and $E_j^n$'s are uniformly bounded in the Lipschitz norm, the Arzela-Ascoli theorem implies that, possibly along a subsequence, one may assume that $S^n(t+t_n)\to \tilde{S}$, $P^n(t+t_n)\to \tilde{P}$ and $E_j^n(t+t_n)\to \tilde{E_j}$ locally uniformly for for $t\in \R$. It remains to deal with the sequences $\{i_j^n\}_n$ with $j=1,2,3$. Denoting by $\tilde {i_1^n}(t,\cdot)= i_1^n(t+t_n,\cdot)$ and using the Volterra formulation \eqref{eq-volterra} it comes 
\begin{equation*}
\tilde {i_1^n}(t,a)= \begin{cases}
i_{10}(a-t+t_n) \frac{D_1(a)}{D_1(a-t+t_n)}, \quad \text{ for } t+t_n<a,\\
\frac{S^n(t-a+t_n)}{P^n(t-a+t_n)}E_1^n(t-a+t_n) D_1(a), \quad \text{ for } t+t_n\ge a.
\end{cases}
\end{equation*}
Since $\frac{S^n(t-a+t_n)}{P^n(t-a+t_n)}E_1^n(t-a+t_n) D_1(a)$ converges locally uniformly towards the function $\frac{\tilde S(t-a)}{\tilde P(t-a)}\tilde{E_1}(t-a) D_1(a)$ as $t_n \to +\infty$, we easily conclude that 
\[
{i_1^n}(t_n,\cdot)= \tilde{i_1^n}(0,\cdot) \to \frac{\tilde S(-\cdot)}{\tilde P(-\cdot)}\tilde{E_1}(-\cdot) D_1(\cdot) \text{ in } L^1(0,\infty,\R).
\]
Similarly, we also find that 
\[
{i_j^n}(t_n,\cdot)= \tilde{i_j^n}(0,\cdot) \to  \tilde{E_j}(-\cdot) D_j(\cdot) \text{ in } L^1(0,\infty,\R)  \text{, and for } j=2,3.
\]
Item 3. follows. 

For item 4. of the theorem, items 2. and 3. show that $\Phi$ is point dissipative, eventually bounded on bounded sets, and asymptotically smooth. Thus, item 4. follows from Theorem 2.33 of \cite{Smith2011}.

To complete the proof of the theorem, it remains to proof Claim \ref{claim_lips}.

\begin{proof}[Proof of Claim \ref{claim_lips}]
Let $t\in\mathbb{R}^+$ and $\eta >0$. Recalling \eqref{eq-E1-E2} and setting $B_1(t)= \int_0^\infty \beta(a) i_1(t,a)\d a$,  $B_2(t)= \varepsilon \int_0^\infty \beta(a) i_2(t,a)\d a$ and $B_3(t)= \delta \int_0^\infty \beta(a) i_3(t,a)\d a$, it comes 
\begin{equation*}
\begin{split}
B_1(t+ \eta)-B_1(t)=  & \int_0^\eta
\beta(a) i_1(t+\eta,a) \d a+ \int_\eta^\infty
\beta(a) i_1(t+\eta,a) \d a-\int_0^\infty \beta(a) i_1(t,a) \d a\\
= & \int_0^\eta
\beta(a) i_1(t+\eta-a,0) D_1(a) \d a+ \int_\eta^\infty
\beta(a) i_1(t+\eta,a) \d a-\int_0^\infty \beta(a) i_1(t,a) \d a.
\end{split}
\end{equation*}
Since the semiflow $\Phi$ is bounded and by Assumption \ref{asym1} (item 2.), we can find $C>0$ such that 
\begin{equation*}
\begin{split}
B_1(t+ \eta)-B_1(t) \le & C \|\beta\|_{\infty}^2 \eta + \int_0^\infty
\beta(a+\eta) i_1(t+\eta,a+\eta) \d a-\int_0^\infty \beta(a) i_1(t,a) \d a.
\end{split}
\end{equation*}
Then, recalling \eqref{eq-volterra} and combining the integrals, we write 
\begin{equation*}
\begin{split}
B_1(t+ \eta)-B_1(t) \le & C \|\beta\|_{\infty}^2 \eta +
\int_0^\infty \beta(a+\eta) \left(e^{-\int_a^{a+\eta}(\gamma_1(\sigma)+d_1(\sigma))\d \sigma}- 1 \right)  i_1(t,a) \d a\\
&+
 \int_0^\infty \left(\beta(a+\eta) - \beta(a) \right) i_1(t,a) \d a.
\end{split}
\end{equation*}
Again by Assumption \ref{asym1} (item 2.), we have $1 \ge e^{-\int_a^{a+\eta}(\gamma_1(\sigma)+d_1(\sigma))\d \sigma}\ge e^{-(\|\gamma_1\|_{\infty}+\|d_1\|_{\infty})\eta} \ge 1- (\|\gamma_1\|_{\infty}+\|d_1\|_{\infty})\eta  $. Therefore, $ \beta(a+\eta) \left|e^{-\int_a^{a+\eta}(\gamma_1(\sigma)+d_1(\sigma))\d \sigma}- 1 \right| \le  \|\beta\|_{\infty} (\|\gamma_1\|_{\infty}+\|d_1\|_{\infty})\eta$ and since the semiflow $\Phi$ is bounded we can find a positive constant $C$ such that  
\begin{equation*}
\begin{split}
B_1(t+ \eta)-B_1(t) \le & C \|\beta\|_{\infty}^2 \eta +C  \|\beta\|_{\infty} (\|\gamma_1\|_{\infty}+\|d_1\|_{\infty})\eta\\
&+
\int_0^\infty \left(\beta(a+\eta) - \beta(a) \right) i_1(t,a) \d a.
\end{split}
\end{equation*}
Next, using Assumption \ref{asym1} (item 3.) and the boundedness of the semiflow, we find a positive constant $C$ such that 
\[
\int_0^\infty \left|\beta(a+\eta) - \beta(a) \right| i_1(t,a) \d a \le C \eta.
\] 
From where, we find $C_1>0$ such that  
\[
\left|B_1(t+ \eta)-B_1(t) \right| \le C_1 \eta.
\]
Using the same arguments, we find $C_2>0$ such that 
\[
\left|B_2(t+ \eta)-B_2(t) \right| \le C_2 \eta.
\]
Since $E_1=B_1+B_2+B_3$, combining the two previous inequalities, it follows that $E_1$ is Lipschitz on $\R_+$. Similarly, $E_2$ and $E_3$ are also Lipschitz and this ends the proof of the claim.
\end{proof}

\section{Proof of Theorem \ref{Theo-asymptotic}} \label{sec-proof-thm24}
First, let us introduce  some useful technical materials by establishing some properties of
the complete solutions of system \eqref{model1}.

\subsection{Technical materials} 
The first result deals with spectral properties of the linearized semiflow $\Phi$ at a given equilibrium point $\varphi^*= \left(v^*,0_{\R^3},u_1^*,u_2^*,u_3^*\right) \in \mathcal{X}_{0^+} $. The linearized system at $\varphi^*$ reads
\[
\frac{d \varphi(t)}{dt}= \left( A+F^*\right) \varphi(t),
\] 
where $F^*$ is a linear bounded operator given by
\[
F^*\left(
v,0_{\R^3},
u_1,u_2,u_3
\right)^T
=\left(\begin{array}{c}
-\mathcal{W}^*,
\mathcal{W}^*,
\int_0^\infty \gamma_1(a) u_1(a) \d a,
\int_0^\infty \gamma_2(a) u_2(a) \d a,
0,0,0
\end{array}\right)^T,
\]
wherein
\[
\begin{split}
\mathcal{W}^*=&\frac{1}{ \left(v^*+  \overline A^*\right)^2} \left[ v \overline A^*  \overline B^* -  v^* \overline B^* \int_0^\infty \sum_{j=1}^3u_j(a) \d a \right.\\
&\left.  +	v^*\left(v^*+  \overline A^*\right) \int_0^\infty \beta(a) \left(u_1(a)+\varepsilon u_2(a) +\delta u_3(a)\right) \d a  \right],
\end{split}
\]
and
\[
\begin{split}
\overline A^*=& \int_0^\infty \sum_{j=1}^3u_j^*(a) \d a,\\
\overline B^*=& \int_0^\infty\beta(a) \left(u_1^*(a)+\varepsilon u_2^*(a) +\delta u_3^*(a)\right) \d a.
\end{split}
\]

\begin{lemma}\label{lem-spectrum}
	Let us set $\Sigma= \left\{ \lambda \in  \C: R_e(\lambda)> -\mu 
	\right\}$. The spectrum $\sigma \left(A+F^* \right) \cap \Sigma $ consists of a point spectrum and one has
	\[
	\sigma \left(A+F^* \right) \cap \Sigma= \left\{ \lambda\in \Sigma: \chi(\lambda,\varphi^*)=0
	\right\}
	\]
	where the function $\chi(\cdot,\varphi^*): \Sigma \to \C$ is defined by 
	\begin{equation}\label{eq-chi}
	\chi(\cdot,\varphi^*)= 1 - \frac{\mathcal{Q}^*(\cdot) (\cdot+\mu)}{\overline A^* \overline B^*+ \left(v^*+  \overline A^* \right)^2(\cdot+\mu)}, 
	\end{equation}
	and wherein 
	\[
	\begin{split}
	\mathcal{Q}^*(\lambda)=& -  v^* \overline B^* \left( D_1^\lambda + D_2^\lambda \Gamma_1(\lambda) + D_3^\lambda  \Gamma_1(\lambda)  \Gamma_2(\lambda)\right)\\
	&+v^*\left(v^*+  \overline A^* \right) \left( \Omega_1(\lambda) + \varepsilon\Omega_2(\lambda) \Gamma_1(\lambda) + \delta \Omega_3(\lambda)  \Gamma_1(\lambda)  \Gamma_2(\lambda)\right) ,\\
	\Omega_j(\lambda)= & \int_{0}^\infty \beta(a) D_j(a)e^{-\lambda a}  \d a,\quad \Gamma_j(\lambda) = \int_0^\infty \gamma_j(a) D_j(a)e^{-\lambda a}  \d a,\\
	D_j^\lambda =& \int_0^\infty  D_j(a)e^{-\lambda a}  \d a.
	\end{split}
	\]
\end{lemma}

\begin{proof}
	Let us denote by $A_0$ the part of $A$ in $\mathcal{X}_0$. Then it is the infinitesimal generator of a $C_0$-semigroup on $\mathcal{X}_0$ denoted by $\{T_{A_0}(t) \}_t$. We can easily check that the essential growth rate of this semigroup satisfies $\omega_{0,\text{ess}} (A_0) \le -\mu$. Since operator $F^*$ is compact, results in \cite{Ducrot2008,Thieme1997} apply and ensure that the essential growth rate of $\{T_{(A+F^*)_0}(t) \}_t$, the $C_0$-semigroup generated by the part of $A+F^*$ in $\mathcal{X}_0$ is such that $\omega_{0,\text{ess}} (A+F^*)_0 \le -\mu$. Applying the results in \cite{Magal2009b}, the latter inequality ensures that $\sigma(A+F^*) \cap \Sigma$ is only composed of a point spectrum of $(A+F^*)$. The derivation of the characteristic equation is standard and we refer to \cite{Chu2009,Magal2010}. Indeed, let us first notice that $A$ is a Hille-Yosida operator on $\mathcal{X}$, meaning that $[-\mu, \infty) \subset \sigma(A)$ and for $ z >-\mu$ one has $\| \left(zI - A\right)^{-1} \| \le \left(z+\mu\right)^{-1}$. Therefore, for $\lambda\in \C$ such that $R_e(\lambda)>-\mu$ it is easily checked that $[\lambda I
	-(A+F^*)]^{-1}=(\lambda I -A)^{-1}[I-F^*(\lambda I -A)^{-1}]^{-1}$.
	That is to say
	\[\lambda\in \sigma(A+F^*) \Leftrightarrow 1\in \sigma(F^*(\lambda I
	-A)^{-1}).\] We also have 
	\begin{equation}\label{eq-charac-alpha}
	\begin{split} 
	&[I-F^*(\lambda I -A)^{-1}] \left( \alpha_1, \alpha_2,\alpha_3,\alpha_4, w_1,w_2,w_3 \right)^T =\left( v,0_{\R^3}, u_1,u_2,u_3 \right)^T \Leftrightarrow\\
	&\begin{cases}
	\alpha_1 + \widetilde{\mathcal{W}}^*= v,\\
	\alpha_2 - \widetilde{\mathcal{W}}^*= 0,\\
	\alpha_3 - \int_0^\infty \gamma_1 \tilde w_1(a) \d a=0,\\
	\alpha_4 - \int_0^\infty \gamma_2 \tilde w_2(a) \d a=0,\\
	w_1=u_1; w_2=u_2; w_3=u_3,
	\end{cases}
	\end{split}
	\end{equation}
	where $ \tilde w_j(a)= \alpha_{j+1} D_j(a)e^{-\lambda a} + g_j(\lambda,a) $; $g_j(\lambda,a)= \int_0^a u_j(s) \frac{D_j(a)}{D_j(s)}e^{-\lambda (a-s)}  \d s$, and 
	\[
	\begin{split}
	\widetilde{\mathcal{W}}^*=& \frac{\alpha_1 \overline A^*  \overline B^*}{ \left(v^*+  \overline A^* \right)^2(\lambda +\mu)} +  \frac{1}{ \left(v^*+  \overline A^* \right)^2} \left[  -  v^* \overline B^* \int_0^\infty \sum_{j=1}^3 \tilde w_j(a) \d a \right.\\
	&\left.  +	v^*\left(v^*+  \overline A^* \right) \int_0^\infty \beta(a) \left(\tilde w_1(a)+\varepsilon \tilde w_2(a) +\delta \tilde w_3(a)\right) \d a  \right].
	\end{split}
	\]
	First, from \eqref{eq-charac-alpha}, we have 
	\[
	\begin{split}
	\alpha_3= &\alpha_2 \Gamma_1(\lambda)  +\int_0^\infty \gamma_1 (a) g_1(\lambda,a) \d a,\\
	\alpha_4=& \alpha_2 \Gamma_1(\lambda)  \Gamma_2(\lambda)  +  \Gamma_2(\lambda)\int_0^\infty \gamma_1 (a) g_1(\lambda,a) \d a +\int_0^\infty \gamma_2 (a) g_2(\lambda,a) \d a.
	\end{split}
	\]
	From where
	\[
	\begin{split}
	& \int_0^\infty \sum_{j=1}^3 \tilde w_j(a) \d a= \left( D_1^\lambda + D_2^\lambda \Gamma_1(\lambda) + D_3^\lambda  \Gamma_1(\lambda)  \Gamma_2(\lambda)\right) \alpha_2\\
	& +D_3^\lambda \left(  \Gamma_2(\lambda)\int_0^\infty \gamma_1 (a) g_1(\lambda,a) \d a +\int_0^\infty \gamma_2 (a) g_2(\lambda,a) \d a\right) \\
	& + D_2^\lambda \int_0^\infty \gamma_1 (a) g_1(\lambda,a) \d a  + \int_0^\infty \sum_{j=1}^3  g_j(\lambda,a)\d a,
	\end{split}
	\]
	and 
	\[
	\begin{split}
	&  \int_0^\infty \beta(a) \left(\tilde w_1(a)+\varepsilon \tilde w_2(a) +\delta \tilde w_3(a)\right) \d a = \left( \Omega_1(\lambda) + \varepsilon\Omega_2(\lambda) \Gamma_1(\lambda) + \delta \Omega_3(\lambda)  \Gamma_1(\lambda)  \Gamma_2(\lambda)\right) \alpha_2\\
	& +\delta \Omega_3(\lambda) \left(  \Gamma_2(\lambda)\int_0^\infty \gamma_1 (a) g_1(\lambda,a) \d a +\varepsilon \int_0^\infty \gamma_2 (a) g_2(\lambda,a) \d a\right) \\
	& +\varepsilon \Omega_2(\lambda) \int_0^\infty \gamma_1 (a) g_1(\lambda,a) \d a  +\int_0^\infty \beta(a) \left( g_1(\lambda,a)+\varepsilon  g_2(\lambda,a) +\delta  g_3(\lambda,a)\right) \d a.
	\end{split}
	\]
	From the tow previous equality, $\widetilde{\mathcal{W}}^*$ rewrites 
	\[
	\begin{split}
	\widetilde{\mathcal{W}}^*=& \frac{\alpha_1 \overline A^*  \overline B^*}{ \left(v^*+  \overline A^* \right)^2(\lambda +\mu)} +  \frac{1}{ \left(v^*+  \overline A^* \right)^2} \left[   \mathcal{Q}^* \alpha_2   +R^*(\lambda,u) \right],
	\end{split}
	\]
	with 
	\[
	\begin{split}
	&R^*(\lambda,u)= -  v^* \overline B^*D_3^\lambda \left(  \Gamma_2(\lambda)\int_0^\infty \gamma_1 (a) g_1(\lambda,a) \d a +\int_0^\infty \gamma_2 (a) g_2(\lambda,a) \d a\right) \\
	& -  v^* \overline B^* D_2^\lambda \int_0^\infty \gamma_1 (a) g_1(\lambda,a) \d a  + \int_0^\infty \sum_{j=1}^3  g_j(\lambda,a)\d a\\
	& +	v^*\left(v^*+  \overline A^* \right)\delta \Omega_3(\lambda) \left(  \Gamma_2(\lambda)\int_0^\infty \gamma_1 (a) g_1(\lambda,a) \d a +\varepsilon \int_0^\infty \gamma_2 (a) g_2(\lambda,a) \d a\right) \\
	& +	v^*\left(v^*+  \overline A^* \right)\varepsilon \Omega_2(\lambda) \int_0^\infty \gamma_1 (a) g_1(\lambda,a) \d a  +\int_0^\infty \beta(a) \left( g_1(\lambda,a)+\varepsilon  g_2(\lambda,a) +\delta  g_3(\lambda,a)\right) \d a.
	\end{split}
	\]
	
	From the two first equations of \eqref{eq-charac-alpha}, a straightforward computation gives 
	\[
	\begin{split}
	&\alpha_1 =- \left(1+ \frac{ \overline A^*  \overline B^*}{ \left(v^*+  \overline A^* \right)^2(\lambda +\mu)}\right)^{-1} \frac{\mathcal{Q}^*(\lambda)\alpha_2}{ \left(v^*+  \overline A^* \right)^2}  + G(\lambda,v,u),\\
	&\alpha_2 \left[1 - \frac{\mathcal{Q}^*(\lambda) (\lambda+\mu)}{\overline A^* \overline B^*+ \left(v^*+  \overline A^* \right)^2(\lambda+\mu)} \right] = \frac{R^*(\lambda,u)}{ \left(v^*+  \overline A^* \right)^2} + \frac{\overline A^* \overline B^* G(\lambda,v,u)}{ \left(v^*+  \overline A^* \right)^2(\lambda+\mu)} 
	\end{split}
	\]
	with 
	\[
	\begin{split}
	G(\lambda,v,u)= \left(1+ \frac{ \overline A^*  \overline B^*}{ \left(v^*+  \overline A^* \right)^2(\lambda +\mu)}\right)^{-1} \left(v - \frac{R^*(\lambda,u)}{ \left(v^*+  \overline A^* \right)^2}\right).
	\end{split}
	\]
	
	By setting 
	\[
	\chi(\lambda, \varphi^*)= 1 - \frac{\mathcal{Q}^*(\lambda) (\lambda+\mu)}{\overline A^* \overline B^*+ \left(v^*+  \overline A^* \right)^2(\lambda+\mu)},
	\]
	we can then isolate $\alpha_2$ (and then $\alpha_1, \alpha_3, \alpha_4$) if and
	only if $ \chi(\lambda, \varphi^*)\neq0$.
\end{proof}

The next results relies on some properties of the complete solutions of system \eqref{model1}.

\begin{lemma} \label{lem-E1-positif}
	Set $y_0=(S_0,i_{10},i_{20},i_{30})$ and let $\left\{ y(t)= \left(S(t),i_1(t,\cdot),i_2(t,\cdot),i_3(t,\cdot)\right)
	\right\}_{t\in \R}\subset \mathcal{X}_{0^+} $ a complete solution of \eqref{model1} passing through $y_0$. Then, $S(t)$ is strictly positive for all $t$ and either $E_1$ is identically zero or $E_1(t)$ is positive for all $t$.
\end{lemma}

\begin{proof}
	Assume that there exists $t_1 \in \R$ such that $S(t_1)=0$. Then, the $S$-equation of \eqref{model1} gives $\dot S(t_1)>0$ meaning that we can find $\delta>0$ sufficiently small such that $S(t_1-\delta)<0$. A contradiction with the fact that the
	total trajectory $y(t)$ lies in $\mathcal{X}_{0^+}$ in for all $t \in \R$.
	
	Next, let us notice that since $y$ is a complete solution, it comes from the Volterra formulation
	\begin{equation*} \label{eq-voltera-complet}
	\begin{split}
	i_1(t,a)=& \frac{S(t-a)}{P(t-a)} E_1(t-a)D_1(a),\\
	i_j(t,a)=& E_j(t-a)D_j(a), j=2,3; \text{ and } \quad \forall (t,a) \in \R \times [0,\infty).
	\end{split}
	\end{equation*}
	From above formulation, we can observe that $i_j(t,a)=i_j(t-a,0)D_j(a)$, $j=1,2,3$. Moreover, $E_{j+1}(t)= \int_0^\infty \gamma_j(a) D_j(a) E_j(t-a) \d a$, with $j=1,2$. Therefore, a straightforward computation gives
	\begin{equation}\label{eq-18}
	\begin{split}
	&E_1(t)= \int_0^\infty \beta(a)D_1(a) \frac{S(t-a)}{P(t-a)} E_1(t-a) \d a\\
	& + \varepsilon \int_0^\infty \int_0^\infty \beta(a)D_2(a) D_1(\eta) \gamma_1(\eta ) \frac{S(t-a-\eta)}{P(t-a-\eta)} E_1(t-a-\eta) \d \eta \d a\\
	& + \delta \int_0^\infty \int_0^\infty \int_0^\infty \beta(a)D_3(a) D_2(\eta) \gamma_2(\eta )D_1(s)\gamma_1(s) \frac{S(t-a-\eta-s)}{P(t-a-\eta-s)} E_1(t-a-\eta-s)\d s \d \eta \d a.
	\end{split}
	\end{equation}
	
	Then, due to Assumption \ref{asym1}, \eqref{eq-18} gives that either $E_1$ is identically zero or $E_1(t)$ takes on positive values for all $t \in \R$.
\end{proof}

\subsection{Proof of Theorem \ref{Theo-asymptotic} item 1} Let $y_0 \in \mathcal{B}$ and let $\Phi(t,y_0) = \left(S(t),i_1(t,\cdot),i_2(t,\cdot),i_3(t,\cdot)\right)$ a complete solution in $\mathcal{B}$ passing through $y_0$ at $t=0$. By setting $\overline{E}_j = \sup_{t \in \R} E_j(t)$ we successively obtain for all $t \in \R$
\[
\begin{split}
E_2(t) \le & \overline{E}_1  \Gamma_1,\\
E_3(t) \le & \overline{E}_1  \Gamma_1 \Gamma_2,\\
E_1(t) \le & \overline{E}_1   \left( \Omega_1 + \varepsilon \Gamma_1 \Omega_2 +\delta \Gamma_1 \Gamma_2 \Omega_3 \ \right).
\end{split}
\]
Taking the supremum on the left-hand side of the $E_1$ inequality it comes $\overline{E}_1 \le \overline{E}_1 \mathcal{R}_0$. Then since $\overline{E}_1$ is non-negative and $\mathcal{R}_0 < 1$ it follows $\overline{E}_1=0$. Similarly, we also find that $\overline{E}_2=\overline{E}_3=0$. Therefore, the attractor is a compact invariant subset of the space $\R\times\{0\} \times \{0\} \times \{0\}$. Since the only such set is the singleton containing the disease-free equilibrium $E^0$, this ends the proof of the first part of item 1 of Theorem \ref{Theo-asymptotic}. 

It remains to prove that $E^0$ is unstable when $\mathcal{R}_0 > 1$. Notice that the function $\chi(\cdot,E^0)= 1 -  \left[ \Omega_1(\cdot) +\varepsilon \Omega_2(\cdot) \Gamma_1(\cdot) +\delta \Omega_3(\cdot) \Gamma_1(\cdot) \Gamma_2(\cdot)\right] $ provided by \eqref{eq-chi} at $E^0$ is a non-decreasing function on $\R_+$ such that $\chi(0,E^0)=1- \mathcal{R}_0<0$ and $\lim_{\lambda \to +\infty} \chi(\lambda,E^0)>0$. As consequence, there exists strictly positive eigenvalue , {\it i.e.}  $\exists \lambda_0 \in \R^*_+$ such that $\chi(\lambda_0,E^0) =0$. This ends the second part of item 1 of Theorem \ref{Theo-asymptotic}.

\subsection{Proof of Theorem \ref{Theo-asymptotic} item 2(i) and 2(ii)}
Let us start by the proof of item 2(i). By Lemma \ref{lem-spectrum} the spectrum of the linearized semiflow $\{\Phi(t,\cdot)\}_t$ at the endemic equilibrium $E^*$ is characterized by the following equation, with $R_e(\lambda)>-\mu$:
\begin{equation}\label{eq-local-sta1}
\begin{split}
&\left( \Omega_1(\lambda) + \varepsilon\Omega_2(\lambda) \Gamma_1(\lambda) + \delta \Omega_3(\lambda)  \Gamma_1(\lambda)  \Gamma_2(\lambda)\right)= 1+ \frac{ \overline A^*}{S^*}+ \frac{ \overline A^*  \overline B^*}{S^*\left(S^*+  \overline A^* \right)}\\
& + \frac{ \overline B^* \left( D_1^\lambda + D_2^\lambda \Gamma_1(\lambda) + D_3^\lambda  \Gamma_1(\lambda)  \Gamma_2(\lambda)\right)}{S^*+  \overline A^*} .
\end{split}
\end{equation}
Since $\overline A^*/S^*= R_0-1$, we obtain from the right hand side of \eqref{eq-local-sta1}
\begin{equation}\label{eq-local-sta2}
\left|\text{RHS of } \eqref{eq-local-sta1} \right|> R_0, \quad \forall \lambda \in \C: R_e(\lambda)>-\mu.
\end{equation}
Now by contradiction let us assume that there exists $\lambda_0 \in \C$ such that $R_e(\lambda_0)> 0$. Then, the left hand side of \eqref{eq-local-sta1} gives 
\begin{equation}\label{eq-local-sta3}
\left|\text{LHS of } \eqref{eq-local-sta1} \right|< \Omega_1(0) + \varepsilon\Omega_2(0) \Gamma_1(0) + \delta \Omega_3(0)  \Gamma_1(0)  \Gamma_2(0)= R_0.
\end{equation}
A contradiction holds from \eqref{eq-local-sta2} and \eqref{eq-local-sta3}. This ends the proof of Theorem \ref{Theo-asymptotic} item 2(i). 

Next, we deal with the proof of Theorem \ref{Theo-asymptotic} item 2(ii). A trivial solution of \eqref{model1} satisfying these initial conditions is given by $i_1(t,\cdot)=i_2(t,\cdot)=i_3(t,\cdot)=0$ where $S$ is such that $\dot S= \Lambda -\mu S$. This trivial solution tends exponentially to $E^0$. Since solutions to the initial value problem are unique, statement 2(ii) of the theorem follows.

\subsection{Proof of Theorem \ref{Theo-asymptotic} item 2(iii)}
We assume that the support of at least one of $i_{j0}$'s  has positive measure, and therefore $E_1(t)=\int_0^\infty \beta(a) \left(i_1(t,a) +\varepsilon i_2(t,a)+ \delta i_3(t,a) \right) \d a$  takes on positive values for arbitrarily large values of $t$ (by Lemma \ref{lem-E1-positif}). Furthermore, Claim \ref{claim_lips} gives that $E_1$ is Lipschitz, it follows that $E_1$ is positive on a set of positive measure. In the sequel, when that exists, we set for a given function $h$: $h^\infty= \limsup_{t \to \infty} h(t)$ and $h_\infty= \liminf_{t \to \infty} h(t)$. 

For $\eta_0>0$, there exists $t_1 \in \R$ such that $E_1(t) \le E_1^\infty + \frac{\eta_0}{2}$ for all $t \ge t_1$. Then, it follows from the $S$-equation of \eqref{model1} that $S_\infty \ge \Lambda \left(\mu +\frac{ E_1^\infty}{P_\infty} +  \frac{\eta_0}{2} \right)^{-1}$. Thus, there exists $t_2\ge t_1$ such that 
\begin{equation}\label{eq-19}
S(t) \ge \Lambda \left(\mu + \frac{ E_1^\infty}{P_\infty} +  \frac{\eta_0}{2} \right)^{-1}, \quad \forall t\ge t_2.
\end{equation}
We perform a time-shift of $t_2$ on the solution being studied, {\it i.e.} we replace the initial condition $y_0$ with $y_1 = \Phi(t_2,y_0)$. The solution passing
through $y_1$ satisfies equations  \eqref{eq-18} and \eqref{eq-19} for all $t$, and the bounds $E_1^\infty$ and $P_\infty>0$ remain valid. Note that \eqref{eq-18} rewrites
\begin{equation*}
\begin{split}
&E_1(t)= \int_0^\infty \beta(a)D_1(a) \frac{S(t-a)}{P(t-a)} E_1(t-a) \d a\\
& + \varepsilon \int_0^\infty \frac{S(t-\tau)}{P(t-\tau)} E_1(t-\tau) \int_0^\tau \beta(a)D_2(a) D_1(\tau-a) \gamma_1(\tau-a )  \d a \d \tau\\
& + \delta \int_0^\infty \frac{S(t-\tau)}{P(t-\tau)} E_1(t-\tau) \int_0^\tau \int_0^{\tau-a} \beta(a)D_3(a) D_2(\eta) \gamma_2(\eta )D_1(\tau-a-\eta)\gamma_1(\tau-a-\eta)  \d \eta \d a \d \tau.
\end{split}
\end{equation*}
From where
\begin{equation}\label{eq-E1-conv}
E_1(t) \ge K \int_0^t u(\tau)   E_1(t-\tau) \d \tau,
\end{equation}

wherein $K=\mu \left(\mu + \frac{E_1^\infty}{P_\infty} +  \frac{\eta_0}{2} \right)^{-1}$ and $u(\tau)=\beta(\tau)D_1(\tau) + \varepsilon \int_0^\tau \beta(a)D_2(a) D_1(\tau-a) \gamma_1(\tau-a ) \d a + \delta \int_0^\tau \int_0^{\tau-a} \beta(a)D_3(a) D_2(\eta) \gamma_2(\eta )D_1(\tau-a-\eta)\gamma_1(\tau-a-\eta)  \d \eta \d a$. Taking the Laplace transform of each side of inequality \eqref{eq-E1-conv} converts the convolution to a product and we obtain for $\lambda \in \C$
\[
K \widehat{u}(\lambda) \widehat{E_1} (\lambda) \le  \widehat{E_1} (\lambda).
\]
Since $E_1$ is positive on a set of positive measure then, $\widehat{E_1}$ is strictly positive and the last inequality gives
\begin{equation}\label{eq-E1-conv2}
\begin{split}
&K \int_0^\infty e^{-\lambda \tau } \left( \beta(\tau)D_1(\tau) + \varepsilon \int_0^\tau \beta(a)D_2(a) D_1(\tau-a) \gamma_1(\tau-a ) \d a \right.\\
&\left. + \delta \int_0^\tau \int_0^{\tau-a} \beta(a)D_3(a) D_2(\eta) \gamma_2(\eta )D_1(\tau-a-\eta)\gamma_1(\tau-a-\eta)  \d \eta \d a \right) \d \tau \le 1.
\end{split}
\end{equation}
Changing the order of integration from the left hand side of \eqref{eq-E1-conv2}, it comes successively 
\[
\begin{split}
\text{LHS of \eqref{eq-E1-conv2}}=& K \int_0^\infty e^{-\lambda \tau }  \beta(\tau)D_1(\tau)\d\tau + K \varepsilon \int_0^\infty \int_0^\infty e^{-\lambda (a+\sigma) }  \beta(a)D_2(a) D_1(\sigma) \gamma_1(\sigma) \d \sigma  \d a \\
&+K\delta \int_0^\infty \int_0^\infty \int_0^\infty  e^{-\lambda (\tau+a+\eta) } \beta(a)D_3(a) D_2(\eta) \gamma_2(\eta )D_1(\tau)\gamma_1(\tau)  \d \tau \d \eta  \d a.
\end{split}
\]
From where \eqref{eq-E1-conv2} rewrites 
\[
\begin{split}
& K \int_0^\infty e^{-\lambda \tau }  \beta(\tau)D_1(\tau)\d\tau + K \varepsilon \int_0^\infty \int_0^\infty e^{-\lambda (a+\sigma) }  \beta(a)D_2(a) D_1(\sigma) \gamma_1(\sigma) \d \sigma  \d a \\
&+K\delta \int_0^\infty \int_0^\infty \int_0^\infty  e^{-\lambda (\tau+a+\eta) } \beta(a)D_3(a) D_2(\eta) \gamma_2(\eta )D_1(\tau)\gamma_1(\tau)  \d \tau \d \eta  \d a \le 1.
\end{split}
\]
Taking limits as $\eta_0$ and $\lambda$ tend to zero in the previous inequality, it comes 
\[
\frac{\mu}{\mu + \frac{E_1^\infty}{P_\infty}} R_0 \le 1 \quad \text{\it i.e. } \frac{E_1^\infty}{P_\infty} \ge \mu \left( \mathcal{R}_0 -1\right).
\]
Then, we have the following.
\begin{proposition}\label{prop-weak-persi}
	If $\mathcal{R}_0>1$, then the semiflow $\{\Phi(t,y_0) \}_{t}$ is uniformly weakly persistent in the sense that there exists $\nu_0>0$ such that 
	\begin{equation*}
	\limsup_ {t \to \infty} \int_0^\infty \beta(a) \left(i_1(t,a) +\varepsilon i_2(t,a)+\delta i_3(t,a) \right) \d a \ge \nu_0.
	\end{equation*}
\end{proposition}

Therefore, the uniform persistence of the semiflow follows from the continuity of $E_1$ (see Claim \ref{claim_lips}), Theorem \ref{thm-semiflow} statement 4, Lemma \ref{lem-E1-positif}, the uniform weakly persistence of the semiflow (see Proposition \ref{prop-weak-persi}) and  Theorem 5.2 in \cite{Smith2011}. This ends the proof of Theorem \ref{Theo-asymptotic} item 2(iii).

\section{Existence of an optimal control} \label{sec-control-existence}

By setting $\mathcal{V}=\{1,1.TF,2,2.TF,3,1.TS,2.TS \} $, let us consider a control $h\in \mathcal{U}$ and denote by $w^h(t,a)= \left(S^h(t), (i_j^h(t,a) )_{j\in \mathcal{V}}  \right)$, (resp. $\lambda^h= \left(\lambda_S^h(t), (\lambda_j^h(t,a) )_{j\in \mathcal{V}} \right)$), the corresponding state, (resp. adjoint), vector variable given by \eqref{model2-compact} and \eqref{adj}-\eqref{adj-cond}. Let us define the map $\mathcal{L}:L^1((0,T_f), \R^3)
\longrightarrow L^\infty((0,T_f),\R^3)$ by $\mathcal{L}(u_1,u_2,u_{3})=(\mathcal{L}_1u_1, \mathcal{L}_2u_2,  \mathcal{L}_{2.TF} u_3)$, where
$$
\mathcal{L}_j u=\left\{\begin{array}{lcll}
0, & \text{ if } &u<0,& \\
u, & \text{ if }& 0\le u<h_j^{\max},& \\
h_j^{\max}, & \text{ if } & u\ge h_j^{\max},& j \in \{1;2;2.TF\}.
\end{array}\right.
$$

By setting ${\mathcal{X}}:=(0,T_f)\times \left(Q_{T_f}\right) ^7$ with $Q_{T_f}=(0,T_f)\times(0,\infty),$ we define the norms $\|\cdot\|_{L^1({\mathcal{X}})}$ and $\|\cdot\|_{L^\infty({\mathcal{X}})}$ such that for a given vector function $(y,x):=\left(y,(x_j)_{j=1,\ldots,7}\right)$, 

\[
\begin{split}
& \|(y,x)\|_{L^1({\mathcal{X}})}=\int_0^{T_f} |y(t)| \d t + \sum_{j=1}^{7} \int_0^{T_f} \int_0^{\infty}  |x_j(t, a)| \d t\d a,\\
& \|(y,x)\|_{L^\infty ({\mathcal{X}})}= \sup_{t \in [0,T_f]}   |y(t)|  + \sum_{j=1}^{7}  \sup_{t \in [0,T_f]}  \int_0^{\infty}  |x_j(t, a)| \d a.
\end{split}
\]

In the same way, define the norms $\|.\|_{L^1(Q_{T_f})}$ and $\|.\|_{L^\infty(Q_{T_f})}$.
We embed our optimal problem in the space $L^1(0,T_f)$  by defining the function
\begin{equation*}
\mathcal{J}(h)=\left\{\begin{array}{clc}
J(h), & \text{if} \;\;h\in\mathcal{U}, \\
+\infty, & \text{if} \;\;h\notin\mathcal{U}.
\end{array}\right.
\end{equation*}

To prove the existence of the optimal control, let us introduce the first preliminary result.

\begin{lemma}\label{lem-control-existence}Let $T_f$ be sufficiently
	small.\\ 1. The map $h\in \mathcal{U} \to w^h \in
	L^1(\mathcal{X})$ is Lipschitz for the norms $\|\cdot \|_{L^1}$ and $\|\cdot \|_{L^\infty}$  in the following ways:
	\begin{equation*}
	\|w^{h}-w^{v}\|_{L^1(\mathcal{X})} \le T_f C \|{h}-{v}\|_{L^1(0,T_f)} \text{ and }
	\|w^{h}-w^{f}\|_{L^\infty(\mathcal{X})} \le T_f C \|{h}-{f}\|_{L^\infty(0,T_f)}
	\end{equation*}
	for all $h,f\in \mathcal{U}$ and where $C>0$ is a constant.\\
	2. For $h\in \mathcal{U}$, the adjoint system
	\eqref{adj}-\eqref{adj-cond} has a weak solution $\lambda^h$ in
	$L^\infty(\mathcal{X})$ such that
	$$
	||\lambda^{h}-\lambda^{f}||_{L^\infty(\mathcal{X})}\le T_f C
	||{h}-{f}||_{L^\infty(0,T_f)},
	$$
	for all $h,f\in \mathcal{U}$ and $C>0$.\\ 3. The functional $\mathcal{J}(h)$ is lower semicontinuous with respect to $L^1(0,T_f)$ convergence.
\end{lemma}

\begin{proof} Recalling $
E_1(t)= \int_0^\infty \beta(a) \left[i_1+i_{1,TF} +\varepsilon( i_2+i_{2,TF}) +\delta i_3\right] (t,a) \d a$, and since the total population $P^\cdot$ is bounded by positive constants, we easily find a constant $C_0>0$ such that 
	\[
	\begin{split}
	& \left|\frac{E_1^{h}(t)}{P^h(t)} - \frac{E_1^{f}(t)}{P^f(t)} \right| \le C_0 \|\beta \|_\infty  \sum_{v \in \{i_1,i_{1,TF}, i_2,i_{2,TF},i_3\}} \|v^h(t,\cdot) - v^f(t,\cdot) \|_ {L^1(0,\infty)}.
	\end{split}
	\]
	Moreover, by using the same arguments as in item 2. of Theorem \ref{thm-semiflow}, we can find a positive constant $M_0$ such that $|S^h(t)| + \sum_{v \in \{i_1,i_{1,TF},i_2,i_{2,TF},i_3\}} \|v^h(t,\cdot) \|_ {L^1(0,\infty)} \le M_0 $ for all $h \in \mathcal{U}$ and a.e. $t$. 
	
	Then, integrating the $S$-component of \eqref{model2} it comes
	\begin{equation}\label{lip-S}
	\begin{split}
	&|S^h(t)- S^f(t) | \le  \|\beta \|_\infty C_0 M_0 \int_0^t  e^{-\mu \tau}    |S^h(\tau) - S^f(\tau)|  \d \tau \\
	&+ \|\beta \|_\infty C_0 M_0  \int_0^t e^{-\mu \tau}   \sum_{v \in \{i_1,i_{1,TF},i_2,i_{2,TF},i_3\}} \|v^h(\tau,\cdot) - v^f(\tau,\cdot) \|_ {L^1(0,\infty)}  \d \tau.
	\end{split}
	\end{equation}
	
	Moreover, the Volterra integral formulation of system
	\eqref{model2-compact} gives for $i_1$-equation 
	\begin{equation*}
	i_1^h(t,a)= \begin{cases}
	i_{10}(a-t) \frac{D_1(a)}{D_1(a-t)}, \quad \text{ for } t<a,\\
	\left(1- h_1(t-a)\right)\frac{ S^h(t-a)}{P^h(t-a)} E_1^{h}(t-a) D_1(a), \quad \text{ for } t\ge a
	\end{cases}
	\end{equation*}
	
	From where,
	\begin{equation}\label{lip-i1}
	\begin{split}
	&\|i_1^h(t,\cdot) - i_1^f(t,\cdot) \|_{L^1(0,\infty)} \le M_0^2 \|\beta\|_\infty\| h_1- f_1\|_{L^1(0,T_f)} \\ 
	& + C_0M_0 T_f \|\beta \|_\infty  \left( \sum_{v \in \{i_1,i_{1,TF},i_2,i_{2,TF},i_3\}} \|v^h(t,\cdot) - v^f(t,\cdot) \|_ {L^1(0,\infty)} + |S^h(t)- S^f(t) |\right) .
	\end{split}
	\end{equation}

	Therefore, applying same arguments for $i_{1,TF}$, $i_2$, $i_{2,TF}$ and $i_3$ as for estimates \eqref{lip-i1} and combining with \eqref{lip-S}, it follows that for $T_f$ sufficiently small,
	$$
	||w^{h}-w^{v}||_{L^1(\mathcal{X})}\le T_f C
	||{h}-{v}||_{L^1(0,T_f)}.
	$$
	The same arguments is then apply for the norm $L^\infty$ and for item 2. It remains to prove item 3.
	
	We suppose that $h_n:=(h_{1n},h_{2n},h_{2n,TF})\to
	h:=(h_1,h_2,h_{2,TF})$ in $L^1(0,T_f)$. Possibly along a
	subsequence and using the same notation, $h_n^2\to h^2$ a.e. on
	$(0,T_f)$ by (see \cite{Evans1992}, p.21). By Lebesgue's
	dominated convergence theorem, it comes $\displaystyle\lim_{n\to \infty}
	||h_n^2||_{L^1(0,T_f)}= ||h^2||_{L^1(0,T_f)}$. We have
	the similar arguments for $||f^2||_{L^1(0,T_f)}$. These
	handle the convergence of the squared terms in the functional.
	
	Next, we illustrate the convergence of one term in the functional,
	\begin{equation*}
	\begin{array}{rl}
	||B\gamma_2 (i_2^{h_n}- i_2^h)||_{L^1(Q_{T_f})} \le &  ||B||_\infty ||\gamma_2||_\infty 
	||w^{h_n}-w^h|| _{L^1(\mathcal{X})}\\
	\le & C T_f ||{h_n}-h|| _{L^1(0,T_f)}.
	\end{array}
	\end{equation*}
	Therefore,
	$$
	\left|\mathcal{J}(h_n)-\mathcal{J}(h)\right| \le C T_f
	||{h_n}-h|| _{L^1(0,T_f)}.
	$$
	From where we deduce the lower semi-continuity, $\mathcal{J}(h)\le
	\displaystyle\liminf_{n\to \infty} \mathcal{J}(h_n)$.
\end{proof}

The functional $\mathcal{J}:(0,T_f) \to (-\infty,\infty]$ is lower semi-continuous with respect to strong
$L^1$ convergence but not with respect to weak $L^1$ convergence.
Thus, in general it does not attain its infimum on $(0,T_f)$.
Thus we circumvent this situation by using the Ekerland's
variational principle (see \cite{Ekeland1974}): for
$\delta>0$, there exists $h_\delta$ in
$L^1(0,T_f)$ such that
\begin{eqnarray}
\mathcal{J}(h_\delta) &\le & \inf _{h\in \mathcal{U}}
\mathcal{J}(h)+\delta,\label{eq-20}\\
\mathcal{J}(h_\delta)&=& \min _{h\in \mathcal{U}}\left\{
\mathcal{J}(h)+
\sqrt{\delta}||h_\delta-h||_{L^1(0,T_f)}\right\}.
\label{eq-21}
\end{eqnarray}
Note that, by \eqref{eq-21}, the perturbed functional
$$
\mathcal{J}_\delta(h)= \mathcal{J}(h)+
\sqrt{\delta}||h_\delta-h||_{L^1(0,T_f)}
$$
attains its infimum at $h_\delta$. By the same argument as in
Section \ref{sec-optimality-condition}, and using the projection
map $\mathcal{L}$ on $\mathcal{U}$, it comes that

\begin{lemma}\label{lem-control-existence1}
	If $h_\delta$ is an optimal control minimizing the functional $\mathcal{J}_\delta(h), $ then
	\[
	h_\delta=\mathcal{L}\Big(\hat{h}_{1}(\lambda^{h_\delta})+\frac{\sqrt{\delta}\pi_1^{\delta}}{2C_1}, 
	\hat{h}_{2}(\lambda^{h_\delta})+\frac{\sqrt{\delta}\pi_2^{\delta}}{2C_2},
	\hat{h}_{2,TF}(\lambda^{h_\delta})+\frac{\sqrt{\delta}\pi_3^{\delta}}{2C_3}\Big);
	\]
	where  $\pi_j^\delta \in L^\infty(0,T_f), $ with $|\pi_j^{\delta}(\cdot)|\le 1$ and 
	\[
	\begin{split}
	\hat{h}_1(\lambda^{h_\delta})=&\displaystyle\frac{\left(\lambda_{i_1}(t,0)- p_1 \lambda_{i_{1,TF}}(t,0) -(1-p_1) \lambda_{i_{1,TS}}(t,0) \right)}{2C_1}\frac{SE_1^{h_\delta}}{P^{h_\delta}},\\
	\hat{h}_2(\lambda^{h_\delta})=&\displaystyle\frac{\left(\lambda_{i_2}(t,0)- p_2 \lambda_{i_{2,TF}}(t,0)-(1-p_2)\lambda_{i_{2,TS}}(t,0) \right)}{2C_2} E_{2}^{h_\delta},\\
	\hat{h}_{2,TF}(\lambda^{h_\delta})=&\displaystyle \frac{\left(\lambda_{i_2}(t,0)- p_{2,TF} \lambda_{i_{2,TF}}(t,0)- (1-p_{2,TF}) \lambda_{i_{2,TS}}(t,0) \right) }{2C_3} E_{2,TF}^{h_\delta}.
	\end{split}
	\]
\end{lemma}

By Lemmas \ref{lem-control-existence} and \ref{lem-control-existence1}, we are now ready to prove the existence and uniqueness of an optimal controller. Namely, we have the following theorem.
Namely, we have the following result.

\begin{theorem}\label{thm-control-existence}
	If $ \frac{T_f}{2}\sum_{j=1}^3 \frac{1}{C_j}$ is sufficiently small, there exists one and only one optimal controller $h^\ast$ in $\mathcal{U}$ minimizing $\mathcal{J}(h)$. 
\end{theorem}

\begin{proof}
	We star with the uniqueness by defining $\mathcal{F}:\mathcal{U}\longrightarrow\mathcal{U}$ by 
	$$
	\mathcal{F}(h)=\mathcal{L}\Big(\hat{h}_{1}( \lambda^h), \hat{h}_{2}( \lambda^h),\hat{h}_{2,TF}(\lambda^h)\Big)
	$$
	wherein $w^h$ and $\lambda^h$ are state and adjoint solutions corresponding to $h$ as in previous sections. Using the Lipschitz properties of $w^h$ and $\lambda^h$ (see Lemma \ref{lem-control-existence}),  for $h, \bar{h}\in\mathcal{U},$ we find $C>0$ such that for $T_f$ sufficiently small
	\begin{eqnarray*}
		\|\mathcal{L}_1(h_{1})-\mathcal{L}_1(\bar{h}_{1})\|_{L^\infty(0,T_f)} 
		&\le&\displaystyle\frac{C_0\max\left(M_0,M_0^2\right)}{2C_1}  \\
		&\times&\left(\|w^h-w^{\bar h} \|_{L^\infty(Q_{T_f})} +  \|\lambda^h-\lambda^{\bar h} \|_{L^\infty(Q_{T_f})}\right) \\
		&\le & \frac{C T_f}{2 C_1} \|h-\bar h \|_{L^\infty(0,T_f)}.
	\end{eqnarray*}
	Applying the same arguments for $\mathcal{L}_2$ and $\mathcal{L}_{2,TF}$, it comes
	\begin{eqnarray}
	\|\mathcal{F}(h)-\mathcal{F}(\bar{h})\|_{L^\infty(0,T_f)} \le CT_f\|h-\bar{h}\|_{L^\infty(0,T_f)}\Big(\frac{1}{2C_1}+\frac{1}{2C_2}+\frac{1}{2C_3}\Big)
	\label{F}
	\end{eqnarray}
	where the positive constant $C$ depends on the $L^\infty$ bounds on the state and the adjoint solutions and Lipschitz constants. Therefore, if $T_f < \frac{2}{C}(\frac{1}{C_1}+\frac{1}{C_2}+\frac{1}{C_3})^{-1},$ thus the map $\mathcal{F}$ 
	has a unique fixed point $h^\ast.$

	It remains to prove that this fixed point is an optimal controller. For that ends, we used the approximate minimizers $h_\delta$ from Ekerland's variational principle. From Lemma \ref{lem-control-existence1} and the contraction property of $\mathcal{F}, $ we have
	
	\begin{equation}\begin{array}{rl}
	& \Big\|\mathcal{F}(h_\delta)-\mathcal{L}\Big(\hat{h}_{1}(\lambda^{h_\delta})+\frac{\sqrt{\delta}\pi_1^{\delta}}{2C_1}, 
	\hat{h}_{2}(\lambda^{h_\delta})+\frac{\sqrt{\delta}\pi_2^{\delta}}{2C_2},
	\hat{h}_{2,TF}(\lambda^{h_\delta})+\frac{\sqrt{\delta}\pi_3^{\delta}}{2C_3}\Big)\Big\|_{L^\infty(0,T_f)}=\\
	&\Big\|\mathcal{L}\Big(\hat{h}_{1}( \lambda^{h_\delta}), \hat{h}_{2}( \lambda^{h_\delta}), \hat{h}_{2,TF}(\lambda^{h_\delta})\Big)\\
	&-\mathcal{L}\Big(\hat{h}_{1}(\lambda^{h_\delta})+\frac{\sqrt{\delta}\pi_1^{\delta}}{2C_1}, 
	\hat{h}_{2}( \lambda^{h_\delta})+\frac{\sqrt{\delta}\pi_2^{\delta}}{2C_2},
	\hat{h}_{2,TF}(\lambda^{h_\delta})+\frac{\sqrt{\delta}\pi_3^{\delta}}{2C_3} \Big)\Big\|_{L^\infty(0,T_f)}\\
	&\leq \sum_{j=1}^3 \Big\|\frac{\sqrt{\delta}\pi_j^{\delta}}{2C_j} \Big\|_{L^\infty(0,T_f)}\leq \sqrt{\delta}\Big(\frac{1}{2C_1}+ \frac{1}{2C_2}+\frac{1}{2C_3}\Big).
	\end{array}
	\label{22}
	\end{equation}
	
	Consequently, from (\ref{F}) and (\ref{22}), we have
	\begin{equation*}
	\begin{array}{ll}
	&  \|h^\ast-h_\delta\|_{L^\infty(0,T_f)}=\\
	&\Big\|\mathcal{F}(h^\ast)-\mathcal{L}\Big(\hat{h}_{1}(\lambda^{h_\delta})+\frac{\sqrt{\delta}\pi_1^{\delta}}{2C_1}, 
	\hat{h}_{2}(\lambda^{h_\delta})+\frac{\sqrt{\delta}\pi_2^{\delta}}{2C_2},
	\hat{h}_{2,TF}( \lambda^{h_\delta})+\frac{\sqrt{\delta}\pi_3^{\delta}}{2C_3} \Big)\Big\|_{L^\infty(0,T_f)}\\
	&\leq
	\|\mathcal{F}(h^\ast)-\mathcal{F}(h_\delta)\|_{L^\infty(0,T_f)}\\
	&+\Big\|\mathcal{F}(h_\delta)-\mathcal{L}\Big(\hat{h}_{1}(\lambda^{h_\delta})+\frac{\sqrt{\delta}\pi_1^{\delta}}{2C_1}, 
	\hat{h}_{2}(\lambda^{h_\delta})+\frac{\sqrt{\delta}\pi_2^{\delta}}{2C_2},
	\hat{h}_{2,TF}(\lambda^{h_\delta})+\frac{\sqrt{\delta}\pi_3^{\delta}}{2C_3}\Big)\Big\|_{L^\infty(0,T_f)}\\
	&\leq \left(C T_f\|h^\ast-h_\delta\|_{L^\infty(0,T_f)}+ \sqrt{\delta}\right) \Big(\frac{1}{2C_1}+\frac{1}{2C_2}+\frac{1}{2C_3}\Big).
	\end{array}
	\end{equation*}
	Since $C T_f\Big(\frac{1}{2C_1}+\frac{1}{2C_2}+\frac{1}{2C_3}\Big)$ is sufficiently small, it comes 
	$$
	\|h^\ast-h_\delta\|_{L^\infty(0,T_f)}\leq \sqrt{\delta}\Big[1-C T_f \sum_{j=1}^3 \frac{1}{2C_j} \Big]^{-1}  \sum_{j=1}^3 \frac{1}{2C_j},
	$$
	which gives $h_\delta\longrightarrow h^\ast$ in $L^\infty(0,T_f)$ (as $\delta\longrightarrow0$) and by \eqref{eq-20} 
	$$
	\mathcal{J}(h^\ast)=\inf_{h\in\mathcal{U}}\mathcal{J}(h).
	$$ 
\end{proof}


\section{Basic reproduction number of model (\ref{model1})} \label{basicreproduction}
Let $N(t)=\frac{S(t)}{P(t)}E_1(t)$ be the density of newly HIV infected at time $t$. Then from (\ref{model1}) one has
$$N(t)=E_1(t)=\int_0^\infty\beta(a)(i_1(t,a)+\varepsilon i_2(t,a) +\delta i_3(t,a))da $$ where $i_j(t,a)$'s are given by the resolution of the linearized system \eqref{model1} at the disease free equilibrium $E^0$. Then the Volterra formulation \eqref{eq-volterra} yields
\begin{equation}\label{compoute-R01}
N(t)= \int_0^t\beta(a) \left(D_1(a)E_1(t-a)+\varepsilon D_2(a)E_2(t-a) +\delta D_3(a)E_3(t-a) \right)\d a +N_0(t),
\end{equation}
with $N_0(t)=  \int_t^\infty \beta(a) \left(i_{10}(a-t) \frac{D_1(a)}{D_1(a-t)} + \varepsilon i_{20}(a-t) \frac{D_2(a)}{D_2(a-t)} + \delta i_{30}(a-t) \frac{D_3(a)}{D_3(a-t)} \right) \d a.$

Further, we have
\[\begin{split}
E_2(t)=& \int_0^t \gamma_1(a) D_1(a)E_1(t-a)\d a + H_2(t),\\
E_3(t)=&\int_0^t \gamma_2(a) D_2(a)E_2(t-a)\d a +\int_t^\infty \gamma_2(a) i_{20}(a-t) \frac{D_2(a)}{D_2(a-t)} \d a\\
=&\int_0^t E_1(t-a) U(a) \d a +H_3(t),
\end{split}
\]
with $H_2(t)= \int_t^\infty \gamma_1(a) i_{10}(a-t) \frac{D_1(a)}{D_1(a-t)} \d a$, $H_3(t)= \int_0^t \gamma_2(a) D_2(a) H_2(t-a)\d a +\int_t^\infty \gamma_2(a) i_{20}(a-t) \frac{D_2(a)}{D_2(a-t)} \d a$, and $U(a)= \int_0^a \gamma_1(a-\sigma) D_1(a-\sigma) \gamma_2(\sigma) D_2(\sigma) \d\sigma$.

Therefore, 
\[
\begin{split}
\int_0^t\beta(a) D_2(a)E_2(t-a) \d a=& \int_0^t E_1(t-a) \int_0^a  D_1(a-\sigma) \gamma_1(a-\sigma)  \beta(\sigma) D_2(\sigma) \d \sigma  \d a\\
& +\int_0^t\beta(a) D_2(a) H_2(t-a) \d a ,
\end{split}
\]
and
	\[
	\begin{split}
	&\int_0^t\beta(a) D_3(a)E_3(t-a) \d a=\int_0^tE_1(t-a)\int_0^a U(a-\sigma) \beta(\sigma) D_3(\sigma)\d\sigma\d a\\
	&+ \int_0^t \beta(a)D_3(a) H_3(t-a) \d a.
	\end{split}\]
Consequently, \eqref{compoute-R01} rewrites 
\begin{equation*}
\begin{split}
N(t)=& \int_0^t K(a) E_1(t-a) \d a +N_0(t)\\
& +\varepsilon\int_0^t\beta(a) D_2(a) H_2(t-a) \d a + \delta\int_0^t \beta(a)D_3(a) H_3(t-a) \d a,
\end{split}
\end{equation*}
wherein 
\[
K(a)= \beta(a)D_1(a)+\varepsilon  \int_0^a  D_1(a-\sigma) \gamma_1(a-\sigma)  \beta(\sigma) D_2(\sigma) \d \sigma  \d a +\delta \int_0^a U(a-\sigma) \beta(\sigma) D_3(\sigma)\d\sigma.
\]	
From the above formulation, the basic reproduction number is calculated as 
\[
	\begin{split}
	R_0=& \int_0^\infty K(a)\d a\\
	=& \int_0^\infty \beta(a)D_1(a)\d a+\varepsilon \int_0^\infty \int_0^a  D_1(a-\sigma) \gamma_1(a-\sigma)  \beta(\sigma) D_2(\sigma) \d \sigma\d a \\
	&+\delta \int_0^\infty \int_0^a   \beta(\sigma) D_3(\sigma) U(a-\sigma) \d\sigma  \d a\\
	=&  \int_0^\infty\beta(a)D_1(a)\d a+ \varepsilon \int_0^\infty \beta(a) D_2(a) \int_a^\infty  D_1(\sigma-a) \gamma_1(\sigma-a)   \d \sigma \d a\\
	&+\delta \int_0^\infty \beta(a) D_3(a)\int_a^\infty U(\sigma-a)\d\sigma \d a  \\
	=&  \int_0^\infty\beta(a)D_1(a)\d a+ \varepsilon \int_0^\infty \beta(a) D_2(a) \d a \int_0^\infty  D_1(\sigma) \gamma_1(\sigma)   \d \sigma  \\
	&+\delta \int_0^\infty \beta(a) D_3(a) \d a \int_0^\infty \gamma_2(\sigma) D_2(\sigma)\d \sigma \int_0^\infty \gamma_1(\tau) D_1(\tau)\d\tau,\\
	=&\Omega_1 +\varepsilon \Gamma_1\Omega_2 +\delta \Gamma_1 \Gamma_2\Omega_3.
	\end{split}
	\]

\end{document}